\documentclass{amsart}
\usepackage{graphicx}
\usepackage{amsmath,amsfonts}
\usepackage{amsthm,amssymb,latexsym}
\usepackage[active]{srcltx}
\usepackage{enumerate}
\newtheorem{theorem}{Theorem}[section]
\newtheorem{corollary}[theorem]{Corollary}
\newtheorem{lemma}[theorem]{Lemma}

\newtheorem*{fact*}{Fact}
\newtheorem{proposition}[theorem]{Proposition}
\theoremstyle{definition}
\newtheorem{definition}[theorem]{Definition}

\newtheorem{remark}[theorem]{Remark}
\numberwithin{equation}{section}

\newcommand{\G}{\mathbb G}

\newcommand{\A}{\mathbb A}
\newcommand{\F}{\mathbb F}
\newcommand{\K}{\mathbb K}

\newcommand{\Pp}{\mathbb P}

\newcommand{\bfs}{\boldsymbol}

\newcommand{\fq}{\F_{\hskip-0.7mm q}}
\newcommand{\fqtwo}{\F_{\hskip-0.7mm q^{2}}}

\newcommand{\fqi}{\F_{\hskip-0.7mm q^i}}
\newcommand{\cfq}{\overline{\F}_{\hskip-0.7mm q}}

\def\ifm#1#2{\relax \ifmmode#1\else#2\fi}

\newcommand{\klk}    {\ifm {,\ldots,} {$,\ldots,$}}

\newcommand{\plp}    {\ifm {+\cdots+} {$+\ldots+$}}

\begin{document}

\title[Distribution of factorization patterns]{The distribution of factorization
patterns on linear families of polynomials over a finite field}
\author[E. Cesaratto et al.]{
%
%\author[E. Cesaratto]{
Eda Cesaratto${}^{1,2}$,
%
%\author[G. Matera]{
Guillermo Matera${}^{1,2}$,
%
%\author[M. P\'erez]{
Mariana P\'erez${}^1$,}

\address{${}^{1}$Instituto del Desarrollo Humano,
Universidad Nacional de Gene\-ral Sarmiento, J.M. Guti\'errez 1150
(B1613GSX) Los Polvorines, Buenos Aires, Argentina}
\email{\{ecesarat,\,gmatera,\,vperez\}@ungs.edu.ar}
\address{${}^{2}$ National Council of Science and Technology (CONICET),
Ar\-gentina}

\thanks{The authors were partially supported by the grants
PIP CONICET 11220090100421, UNGS 30/3180 and STIC-AmSud 13STIC-02
``Dynalco''.}%

\subjclass{12E05, 11T06, 12E20, 11G25, 14G05, 14G15, 14B05}
\keywords{Finite fields, factorization patterns, symmetric
polynomials, singular complete intersections, rational points}%

\date{\today}
\maketitle

\begin{abstract}
We obtain estimates on the number $|\mathcal{A}_{\bfs\lambda}|$ of
elements on a linear family $\mathcal{A}$ of monic polynomials of
$\fq[T]$ of degree $n$ having factorization pattern
$\bfs\lambda:=1^{\lambda_1}2^{\lambda_2}\cdots n^{\lambda_n}$. We
show that $|\mathcal{A}_{\bfs\lambda}|=
\mathcal{T}(\bfs\lambda)\,q^{n-m}+\mathcal{O}(q^{n-m-{1}/{2}})$,
where $\mathcal{T}(\bfs\lambda)$ is the proportion of elements of
the symmetric group of $n$ elements with cycle pattern $\bfs\lambda$
and $m$ is the codimension of $\mathcal{A}$. Furthermore, if the
family $\mathcal{A}$ under consideration is ``sparse'', then
$|\mathcal{A}_{\bfs\lambda}|=
\mathcal{T}(\bfs\lambda)\,q^{n-m}+\mathcal{O}(q^{n-m-{1}})$. Our
estimates hold for fields $\fq$ of characteristic greater than 2. We
provide explicit upper bounds for the constants underlying the
$\mathcal{O}$--notation in terms of $\bfs\lambda$ and $\mathcal{A}$
with ``good'' behavior. Our approach reduces the question to
estimate the number of $\fq$--rational points of certain families of
complete intersections defined over $\fq$. Such complete
intersections are defined by polynomials which are invariant under
the action of the symmetric group of permutations of the
coordinates. This allows us to obtain critical information
concerning their singular locus, from which precise estimates on
their number of $\fq$--rational points are established.
\end{abstract}

%
%---------------------------------------------------------------------
%---------------------------------------------------------------------
%---------------------------------------------------------------------
%---------------------------------------------------------------------
%---------------------------------------------------------------------
%---------------------------------------------------------------------
%---------------------------------------------------------------------
%---------------------------------------------------------------------
%
\section{Introduction}
Let $\fq$ be the finite field of $q:=p^s$ elements, where $p$ is a
prime number, and let $\cfq$ denote its algebraic closure. Let $T$
be an indeterminate over $\cfq$ and $\fq[T]$ the set of polynomials
in $T$ with coefficients in $\fq$. Let $n$ be a positive integer and
$\mathcal{P}:={\mathcal P}_n$ the set of all monic polynomials in
$\fq[T]$ of degree $n$. Let $\lambda_1,\cdots,\lambda_n$ be
nonnegative integers such that
$$\lambda_1+2\lambda_2+\cdots+n\lambda_n=n.$$
We denote by ${\mathcal P}_{\bfs \lambda}$ the set of elements
$\mathcal{P}$ with factorization pattern $\bfs
\lambda:=1^{\lambda_1}2^{\lambda_2}\cdots n^{\lambda_n}$, namely the
elements $f\in \mathcal{P}$ which have exactly $\lambda_i$ monic
irreducible factors over $\fq$ of degree $i$ (counted with
multiplicity) for $1\le i\le n$. We shall further use the notation
$\mathcal{S}_{\bfs{\lambda}}:=\mathcal{S}\cap\mathcal{P}_{\bfs
\lambda}$ for any subset $\mathcal{S}\subset\mathcal{P}$.

%A classical result by Gauss asserts that the number of irreducible
%elements of $\mathcal{P}$ is approximately $\frac{1}{n}q^n$, namely
%$|\mathcal{P}_{\bfs\lambda}|\sim\frac{1}{n}|\mathcal{P}|$ for
%$\bfs\lambda:=(0\klk 0,1)$.
In \cite{Cohen70} it was noted that the proportion of elements of
$\mathcal{P}_{\bfs\lambda}$ in $\mathcal{P}$ is roughly the
proportion $\mathcal{T}({\bfs{\lambda}})$ of permutations with cycle
pattern $\bfs \lambda$ in the $n$th symmetric group $\mathbb{S}_n$.
More precisely, it was shown that
\begin{equation}\label{eq: intro: Cohen}
|\mathcal{P}_{\bfs\lambda}|=\mathcal{T}(\bfs\lambda)\,q^n+
\mathcal{O}(q^{n-\frac{1}{2}}),
\end{equation}
where the constant underlying the $\mathcal{O}$--notation depends
only on $\bfs \lambda$. A permutation of $\mathbb{S}_n$ has cycle
pattern $\bfs \lambda$ if it has exactly $\lambda_i$ cycles of
length $i$ for $1\le i\le n$. Observe that
$$\mathcal{T}({\bfs{\lambda}}):=\frac{1}{w({\bfs \lambda})},\quad
w({\bfs \lambda}):=1^{\lambda_1}2^{\lambda_2}\dots n^{\lambda_n}
\lambda_1!\lambda_2!\dots\lambda_n!.$$
In particular, $n!/w({\bfs \lambda})$ is the number of permutations
in $\mathbb{S}_n$ with cycle pattern $\bfs\lambda$.

Furthermore, in \cite{Cohen72} a subset
$\mathcal{S}\subset\mathcal{P}_{\bfs\lambda}$ is called {\em
uniformly distributed} if the proportion
$|\mathcal{S}_{\bfs\lambda}|/|\mathcal{S}|$ is roughly
$\mathcal{T}(\bfs\lambda)$ for every factorization pattern
$\bfs\lambda$. The main result of this paper (\cite[Theorem
3]{Cohen72}) provides a criterion for a linear family of polynomials
of $\mathcal{P}$ to be uniformly distributed in the sense above. As
a particular case we have the classical case of polynomials with
prescribed coefficients, where simpler conditions are obtained (see
\cite[Theorem 1]{Cohen72}; see also \cite{Stepanov87}).

A difficulty with \cite[Theorem 3]{Cohen72} is that the hypotheses
for a linear family of $\mathcal{P}$ to be uniformly distributed
seem complicated and not easy to verify. In fact, in \cite{GaHoPa99}
it is asserted that ``more work need to be done to simplify Cohen's
conditions''. A second concern is that \cite[Theorem 3]{Cohen72}
imposes restrictions on the characteristic $p$ of $\fq$ which may
inhibit its application to fields of small characteristic. Finally,
we are also interested in finding explicit estimates, namely an
explicit admissible expression for the constant underlying (\ref{eq:
intro: Cohen}).

In this paper we consider the linear families in ${\mathcal P}$ that
we now describe. Let $m$, $r$ be positive integers with $3\le r\le
n-m$, let $A_r\klk A_{n-1}$ be indeterminates over $\fq$, and let be
given linear forms $L_1,\ldots,L_m$ of $\fq[A_r\klk A_{n-1}]$ which
are linearly independent and
$\bfs\alpha:=(\alpha_1\klk\alpha_m)\in\fq^m$. Set
$\bfs{L}:=(L_1,\ldots,L_m)$ and define
$\mathcal{A}:=\mathcal{A}(\bfs{L},\bfs\alpha)$ as
\begin{equation}\label{eq: intro: definition A}
\mathcal{A}:=\left\{T^n+a_{n-1}T^{n-1}\plp a_0\in\fq[T]: \bfs
L(a_r\klk a_{n-1})+\bfs \alpha=\bfs 0\right\}.
\end{equation}
Our main results assert that any such family $\mathcal{A}$ is
uniformly distributed. More precisely, we have the following result.
\begin{theorem}\label{theorem: intro: main}
Let $\mathcal{A}_{\bfs \lambda}:=\mathcal{A}
\cap\mathcal{P}_{\bfs\lambda}$. If $p>2$, $q>n$ and $3\le r\le n-m$,
then
\begin{equation}\label{eq: intro: estimate FP I}
\big||\mathcal{A}_{\bfs
\lambda}|-\mathcal{T}(\bfs\lambda)\,q^{n-m}\big|\le
q^{n-m-1}\big(2\,\mathcal{T}(\bfs\lambda)\,D_{\bfs L} \delta_{\bfs
L}q^{\frac{1}{2}} +19\,\mathcal{T}(\bfs\lambda)\,D_{\bfs L}^2
\delta_{\bfs L}^2+n(n-1)\big).
\end{equation}
On the other hand, if $q>n$ and $m+2\le r\le n-m$, then
\begin{equation}\label{eq: intro: estimate FP II}
\big||\mathcal{A}_{\bfs
\lambda}|-\mathcal{T}(\bfs\lambda)\,q^{n-m}\big|\le
q^{n-m-1}\big(21\,\mathcal{T}(\bfs\lambda)\,D_{\bfs L}^3\delta_{\bfs
L}^2 + n(n-1)\big).
\end{equation}
\end{theorem}
Here $\delta_{\bfs L}$ and $D_{\bfs L}$ are certain explicit
discrete invariants associated to the linear variety $\bfs L$ under
consideration. We have the worst--case upper bounds $\delta_{\bfs
L}\le (n-3)!/(n-m-3)!$ and $D_{\bfs L}\le m(n-2)$.

It might be worthwhile to explicitly state what Theorem
\ref{theorem: intro: main} asserts when the family $\mathcal{A}$ of
\eqref{eq: intro: definition A} consists of the polynomials of
$\mathcal{P}$ with certain prescribed coefficients. More precisely,
given $0< i_1<i_2<\cdots< i_m\le n$ and $\bfs
\alpha:=(\alpha_{i_1}\klk \alpha_{i_m})\in\fq^m$, set
$\mathcal{I}:=\{i_1\klk i_m\}$ and
$$\mathcal{A}^m:= \mathcal{A}^m(\mathcal{I},\bfs \alpha):=\left\{T^n+a_1T^{n-1}\plp
a_n\in\fq[T]: a_{i_j}=\alpha_{i_j}\ (1\le j\le m)\right\}.$$
Let $\delta_{\mathcal{I}}:=i_1\cdots i_m$ and
$D_{\mathcal{I}}:=\sum_{j=1}^m(i_j-1)$.
%We shall assume that, if $i_m=n$, then $\alpha_n\not=0$.
We have the following result.
\begin{theorem}\label{theorem: intro: presc coeff}
If $p>2$, $q>n$ and $i_m\le n-3$, then
$$%\begin{equation}\label{eq: intro: estimate FP I}
\big||\mathcal{A}^m_{\bfs
\lambda}|-\mathcal{T}(\bfs\lambda)\,q^{n-m}\big|\le
q^{n-m-1}\big(2\,\mathcal{T}(\bfs\lambda)\,D_{\mathcal{I}}\,
\delta_{\mathcal{I}}\,q^{\frac{1}{2}}
+19\,\mathcal{T}(\bfs\lambda)\,D_{\mathcal{I}}^2\,
\delta_{\mathcal{I}}^2+n(n-1)\big). $$%\end{equation}
On the other hand, for $q>n$ and $i_m\le n-m-2$, we have
$$%\begin{equation}\label{eq: intro: estimate FP II}
\big||\mathcal{A}^m_{\bfs
\lambda}|-\mathcal{T}(\bfs\lambda)\,q^{n-m}\big|\le
q^{n-m-1}\big(21\,\mathcal{T}(\bfs\lambda)\,D_{\mathcal{I}}^3\,
\delta_{\mathcal{I}}^2 + n(n-1)\big). $$%\end{equation}
\end{theorem}

Theorem \ref{theorem: intro: main} strengthens \eqref{eq: intro:
Cohen} in several aspects. First of all, the hypotheses on the
linear families $\mathcal{A}$ in the statement of Theorem
\ref{theorem: intro: main} are relatively wide and easy to verify.
On the other hand, our results are valid either for $p>2$ or without
any restriction on the characteristic $p$ of $\fq$, while \eqref{eq:
intro: Cohen} requires that $p$ is large enough. A third aspect it
is worth mentioning is that \eqref{eq: intro: estimate FP II} shows
that $|\mathcal{A}_{\bfs
\lambda}|=\mathcal{T}(\bfs\lambda)\,q^{n-m}+\mathcal{O}(q^{n-m-1})$,
while \eqref{eq: intro: Cohen} only asserts that $|\mathcal{A}_{\bfs
\lambda}|=\mathcal{T}(\bfs\lambda)\,q^{n-m}+\mathcal{O}(q^{n-m-\frac{1}{2}})$.
Finally, both \eqref{eq: intro: estimate FP I} and \eqref{eq: intro:
estimate FP II} provide explicit expressions for the constants
underlying the $\mathcal{O}$--notation in \eqref{eq: intro: Cohen}
with a good behavior.

In order to prove Theorem \ref{theorem: intro: main}, we express the
number $|\mathcal{A}_{\bfs\lambda}|$ of polynomials in $\mathcal{A}$
with factorization pattern $\bfs\lambda$ in terms of the number of
$\fq$--rational solutions with pairwise--distinct coordinates of a
system $\{R_1=0\klk R_m=0\}$, where $R_1\klk R_m$ are certain
polynomials in $\fq[X_1,\ldots,X_n]$. A critical point for our
approach is that, up to a linear change of coordinates, $R_1\klk
R_m$ are symmetric polynomials, namely invariant under any
permutation of $X_1,\ldots, X_n$. More precisely, we prove that each
$R_j$ can be expressed as a polynomial in the first $n-r$ elementary
symmetric polynomials of $\fq[X_1,\ldots, X_n]$ (Corollary
\ref{coro: fact patterns: systems pattern lambda}). This allows us
to establish a number of facts concerning the geometry of the set
$V$ of solutions of such a polynomial system (see, e.g., Theorems
\ref{theorem: geometry: dimension singular locus}, \ref{theorem:
geometry: proj closure of V is abs irred} and \ref{theorem:
prescribed coeff: complete int and sing locus} and Corollary
\ref{coro: prescribed coeff: complete int and sing locus}).
Combining these results with estimates on the number of
$\fq$--rational points of singular complete intersections of
\cite{CaMaPr13}, we obtain our main results (Theorems \ref{theorem:
estimates: bound with discr locus 1} and \ref{theorem: prescribed
coeff: estimate 2}).

Our methodology differs significantly from that employed in
\cite{Cohen70} and \cite{Cohen72}, as we express
$|\mathcal{A}_{\bfs\lambda}|$ in terms of the number of
$\fq$--rational points of certain singular complete intersections
defined over $\fq$. In \cite[Problem 2.2]{GaHoPa99}, the authors ask
for estimates on the number of elements of $\mathcal{P}$, with a
given factorization pattern, lying in nonlinear families of
polynomials parameterized by an affine variety defined over $\fq$.
As a consequence of general results by \cite{ChDrMa92} and
\cite{FrHaJa94}, it is known that
$|\mathcal{A}_{\bfs\lambda}|=\mathcal{O}(q^r)$, where $r$ is the
dimension of the parameterizing affine variety under consideration.
Nevertheless, very little is known on the asymptotic behavior of
$|\mathcal{A}_{\bfs\lambda}|$ as a power of $q$ and of the size of
the constant underlying the $\mathcal{O}$--notation. We think that
our methods may be extended to deal with this more general case, at
least for certain classes of parameterizing affine varieties.

%
%---------------------------------------------------------------------
%---------------------------------------------------------------------
%---------------------------------------------------------------------
%---------------------------------------------------------------------
%---------------------------------------------------------------------
%---------------------------------------------------------------------
%---------------------------------------------------------------------
%---------------------------------------------------------------------
%
\section{Factorization patterns and roots}
\label{section: fact patterns and roots}
As before, let $n$ be a positive integer with $q>n$ and let
$\mathcal{P}$ be the set of monic polynomials of $\fq[T]$ of degree
$n$. Let $\mathcal{A}\subset\mathcal{P}$ be the linear family
defined in \eqref{eq: intro: definition A} and
$\bfs\lambda:=1^{\lambda_1}\cdots n^{\lambda_n}$ a factorization
pattern. In this section we show that the number
$|\mathcal{A}_{\bfs\lambda}|$ can be expressed in terms of the
number of common $\fq$--rational zeros of certain polynomials
$R_1\klk R_m\in\fq[X_1\klk X_n]$.

For this purpose, let $f$ be an arbitrary element of $\mathcal{P}$
and let $g\in \fq[T]$ be a monic irreducible factor of $f$ of degree
$i$. Then $g$ is the minimal polynomial of a root $\alpha$ of $f$
with $\fq(\alpha)=\fqi$. Denote by $\mathbb G_i$ the Galois group
$\mbox{Gal}(\fqi,\fq)$ of $\fqi$ over $\fq$. Then we may express $g$
in the following way:
$$g=\prod_{\sigma\in\mathbb G_i}(T-\sigma(\alpha)).$$
Hence, each irreducible factor $g$ of $f$ is uniquely determined by
a root $\alpha$ of $f$ (and its orbit under the action of the Galois
group of $\cfq$ over $\fq$), and this root belongs to a field
extension of $\fq$ of degree $\deg g$. Now, for a polynomial
$f\in\mathcal{P}_{\bfs \lambda}$, there are $\lambda_1$ roots of $f$
in $\fq$, say $\alpha_1,\dots,\alpha_{\lambda_1}$ (counted with
multiplicity), which are associated with the irreducible factors of
$f$ in $\fq[T]$ of degree 1; it is also possible to choose
$\lambda_2$ roots of $f$ in $\fqtwo\setminus\fq$ (counted with
multiplicity), say $\alpha_{\lambda_1+1},\dots,
\alpha_{\lambda_1+\lambda_2}$, which are associated with the
$\lambda_2$ irreducible factors of $f$ of degree 2, and so on. From
now on we shall assume that a choice of $\lambda_1\plp\lambda_n$
roots $\alpha_1\klk\alpha_{\lambda_1 \plp\lambda_n}$ of $f$ in
$\cfq$ is made in such a way that each monic irreducible factor of
$f$ in $\fq[T]$ is associated with one and only one of these roots.

Our aim is to express the factorization of $f$ into irreducible
factors in $\fq[T]$ in terms of the coordinates of the chosen
$\lambda_1\plp \lambda_n$ roots of $f$ with respect to suitable
bases of the corresponding extensions $\fq\hookrightarrow\fqi$ as
$\fq$--vector spaces. For this purpose, we express the root
associated with each irreducible factor of $f$ of degree $i$ in a
normal basis $\Theta_i$ of the field extension $\fq\hookrightarrow
\fqi$.

Let $\theta_i\in \fqi$ be a normal element and let $\Theta_i$ be the
normal basis of $\fq\hookrightarrow\fqi$ generated by $\theta_i$,
namely
$$\Theta_i=\left \{\theta_i,\cdots, \theta_i^{q^{i-1}}\right\}.$$
Observe that the Galois group $\mathbb G_i$ is cyclic and the
Frobenius map $\sigma:\fqi\to\fqi$, $\sigma(x):=x^q$ is a generator
of $\mathbb{G}_i$. Thus, the coordinates in the basis $\Theta_i$ of
all the elements in the orbit of a root $\alpha_k\in\fqi$ of an
irreducible factor of $f$ of degree $i$ are the cyclic permutations
of the coordinates of $\alpha_k$ in the basis $\Theta_i$.

%The coordinate vector  of each element $\alpha $ of $\fqi$ in   base
%$\Theta_i$ is a vector  in $\fq^i$.
The vector that gathers the coordinates of all the roots
$\alpha_1\klk\alpha_{\lambda_1+\dots+\lambda_n}$ we have chosen to
represent the irreducible factors of $f$ in the normal bases
$\Theta_1\klk \Theta_n$ is an element of $\fq^n$, which is denoted
by ${\bfs x}:=(x_1,\dots,x_n)$. Set
\begin{equation}\label{eq: fact patterns: ell_ij}
\ell_{i,j}:=\sum_{k=1}^{i-1}k\lambda_k+(j-1)\,i
\end{equation}
for $1\le j \le \lambda_i$ and $1\le i \le n$. Observe that the
vector of coordinates of a root
$\alpha_{\lambda_1\plp\lambda_{i-1}+j}\in\fqi$ is the sub-array
$(x_{\ell_{i,j}+1},\dots,x_{\ell_{i,j}+i})$ of $\bfs x$. With this
notation, the $\lambda_i$ irreducible factors of $f$ of degree $i$
are the polynomials
\begin{equation}\label{eq: fact patterns: gij}g_{i,j}=\prod_{\sigma\in\mathbb G_i}
\Big(T-\big(x_{\ell_{i,j}+1}\sigma(\theta_i)+\dots+
x_{\ell_{i,j}+i}\sigma(\theta_i^{q^{i-1}})\big)\Big)
\end{equation}
for $1\le j \le \lambda_i$. In particular,
\begin{equation}\label{eq: fact patterns: f factored with g_ij}
f=\prod_{i=1}^n\prod_{j=1}^{\lambda_i}g_{i,j}.
\end{equation}

Let $X_1\klk X_n$ be indeterminates over $\cfq$, set $\bfs
X:=(X_1,\dots,X_n)$ and consider the polynomial $G\in
\fq[\bfs{X},T]$ defined as
\begin{equation}\label{eq: fact patterns: pol G}
G:=\prod_{i=1}^n\prod_{j=1}^{\lambda_i}G_{i,j},\quad
G_{i,j}:=\prod_{\sigma\in\mathbb G_i}
\Big(T-\big(X_{\ell_{i,j}+1}\sigma(\theta_i)+
\dots+X_{\ell_{i,j}+i}\sigma(\theta_i^{q^{i-1}})\big)\Big),
\end{equation}
where the $\ell_{i,j}$ are defined as in \eqref{eq: fact patterns:
ell_ij}. Our previous arguments show that an element
$f\in\mathcal{P}$ has factorization pattern ${\bfs \lambda}$ if and
only if there exists $\bfs x\in\fq^n$ with $f=G({\bfs x},T)$.

Next we discuss how many elements $\bfs x\in\fq^n$ yield an
arbitrary polynomial $f=G(\bfs x,T)\in\mathcal{P}_{\bfs\lambda}$.
For $\alpha\in\fqi$, we have that $\fq(\alpha)=\fqi$ if and only if
its orbit under the action of the Galois group $\G_i$ has exactly
$i$ elements. In particular, if $\alpha$ is expressed by its
coordinate vector $\bfs x\in\fq^i$ in the normal basis $\Theta_i$,
then the coordinate vectors of the elements of the orbit of $\alpha$
form a cycle of length $i$, because ${\mathbb G}_i$ permutes
cyclically the coordinates. As a consequence, there is a bijection
between cycles of length $i$ in $\fq^i$ and elements $\alpha\in\fqi$
with $\fq(\alpha)=\fqi$.

In this setting, the notion of an array of type $\bfs \lambda$ will
prove to be useful.
\begin{definition}\label{def: fact patterns: type lambda}
Let $\ell_{i,j}$ $(1\le i\le n,\ 1\le j\le\lambda_i)$ be defined as
in \eqref{eq: fact patterns: ell_ij}. An element ${\bfs
x}=(x_1,\dots, x_n)\in\fq^n$ is said to be of type $\bfs \lambda$ if
and only if each sub-array $\bfs x_{i,j}:=(x_{\ell_{i,j}+1},
\dots,x_{\ell_{i,j}+i})$ is a cycle of length $i$.
\end{definition}
The next result relates $\mathcal{P}_{\bfs \lambda}$ with the set of
elements of $\fq^n$ of type $\bfs\lambda$.
\begin{lemma}
\label{lemma: fact patterns: G(x,T) with fact pat lambda} For any
${\bfs x}=(x_1,\dots, x_n)\in \fq^n$, the polynomial $f:=G({\bfs x},T)$
has factorization pattern $\bfs \lambda$ if and only if ${\bfs x}$
is of type $\bfs \lambda$. Furthermore, for each square--free
polynomial $f\in \mathcal{P}_{\bfs \lambda}$ there are $w({\bfs
\lambda}):=\prod_{i=1}^n i^{\lambda_i}\lambda_i!$ different ${\bfs
x}\in \fq^n $ with $f=G({\bfs x},T)$.
\end{lemma}

\begin{proof}
Let $\Theta_1,\dots,\Theta_n$ be the normal bases introduced before.
Each $\bfs x\in \fq^n$ is associated with a unique finite sequence
of elements $\alpha_k$ $(1\le k\le \lambda_1+\dots+\lambda_n)$ as
follows: each $\alpha_{\lambda_1\plp\lambda_{i-1}+j}$ with $1\le
j\le\lambda_i$ is the element of $\fqi$ whose coordinate vector in
the basis $\Theta_i$ is the sub-array $(x_{\ell_{i,j}+1},
\dots,x_{\ell_{i,j}+i})$ of $\bfs x$.

Suppose that $G({\bfs x},T)$ has factorization pattern $\bfs
\lambda$ for a given $\bfs x\in\fq^n$. Fix $(i,j)$ with $1\le i\le
n$ and $1\le j\le\lambda_i$. Then $G({\bfs x},T)$ is factored as in
\eqref{eq: fact patterns: gij}--\eqref{eq: fact patterns: f factored
with g_ij}, where each $g_{i,j}\in\fq[T]$ is irreducible, and hence
$\fq(\alpha_{\lambda_1\plp\lambda_{i-1}+j})=\fqi$. We conclude that
the sub-array $(x_{\ell_{i,j}+1}, \dots,x_{\ell_{i,j}+i})$ defining
$\alpha_{\lambda_1\plp\lambda_{i-1}+j}$  is a cycle of length $i$.
This proves that $\bfs x$ is of type $\bfs\lambda$.

On the other hand, assume that we are given $\bfs x\in\fq^n$ of type
$\bfs \lambda$ and fix $(i,j)$ with $1\le i\le n$ and $1\le
j\le\lambda_i$. Then
$\fq(\alpha_{\lambda_1\plp\lambda_{i-1}+j})=\fqi$, because the
sub-array $(x_{\ell_{i,j}+1}, \dots,x_{\ell_{i,j}+i})$ is a cycle of
length $i$ and thus the orbit of
$\alpha_{\lambda_1\plp\lambda_{i-1}+j}$ under the action of $\mathbb
G_i$ has $i$ elements. This implies that the factor $g_{i,j}$ of
$G({\bfs x},T)$ defined as in (\ref{eq: fact patterns: gij}) is
irreducible of degree $i$. We deduce that $f:=G(\bfs x,T)$ has
factorization pattern $\bfs\lambda$.

Furthermore, for $\bfs x\in\fq^n$ of type $\bfs\lambda$, the
polynomial $f:=G(\bfs x,T)\in\mathcal{P}_{\bfs\lambda}$ is
square--free if and only if all the roots
$\alpha_{\lambda_1\plp\lambda_{i-1}+j}$ with $1\le j\le \lambda_i$
are pairwise--distinct, non--conjugated elements of $\fqi$. This
implies that no cyclic permutation of a sub-array
$(x_{\ell_{i,j}+1}, \dots,x_{\ell_{i,j}+i})$ with $1\le
j\le\lambda_i$ agrees with another cyclic permutation of another
sub-array $(x_{\ell_{i,j'}+1}, \dots,x_{\ell_{i,j'}+i})$. As cyclic
permutations of any of these sub-arrays and permutations of these
sub-arrays yield elements of $\fq^n$ associated with the same
polynomial $f$, we conclude that there are $w({\bfs
\lambda}):=\prod_{i=1}^n i^{\lambda_i}\lambda_i!$ different elements
$\bfs x\in\fq^n$ with $f=G(\bfs x,T)$.
\end{proof}
%
%---------------------------------------------------------------------
%---------------------------------------------------------------------
%---------------------------------------------------------------------
%---------------------------------------------------------------------
%
\subsection{$G$ in terms of the elementary symmetric polynomials}
Consider the polynomial $G$ of \eqref{eq: fact patterns: pol G} as
an element of $\fq[\bfs X][T]$. We shall express the coefficients of
$G$ by means of the vector of linear forms $\bfs Y:=(Y_1\klk
Y_n)\in\cfq[\bfs X]$ defined in the following way:
\begin{equation}\label{eq: fact patterns: def linear forms Y}
(Y_{\ell_{i,j}+1},\dots,Y_{\ell_{i,j}+i})^{t}:=A_{i}\cdot
(X_{\ell_{i,j}+1},\dots, X_{\ell_{i,j}+i})^{t} \quad(1\le j\le
\lambda_i,\ 1\le i\le n),
\end{equation}
where $A_i\in\fqi^{i\times i}$ is the matrix
$$A_i:=\left(\sigma(\theta_i^{q^{h}})\right)_{\sigma\in {\mathbb G}_i,\ 1\le h\le i}.$$
According to (\ref{eq: fact patterns: pol G}), we may express the
polynomial $G$ as
$$G=\prod_{i=1}^n\prod_{j=1}^{\lambda_i}\prod_{k=1}^i(T-Y_{\ell_{i,j}+k})=
\prod_{k=1}^n(T-Y_k)=T^n+\sum_{k=1}^n(-1)^k\,(\Pi_k(\bfs Y))\,
T^{n-k},$$
where $\Pi_1(\bfs Y)\klk \Pi_n(\bfs Y)$ are the elementary symmetric
polynomials of $\fq[\bfs Y]$. By the expression of $G$ in \eqref{eq:
fact patterns: pol G} we deduce that $G$ belongs to $\fq[{\bfs
X},T]$, which in particular implies that $\Pi_k(\bfs Y)$ belongs to
$\fq[{\bfs X}]$ for $1\le k\le n$. Combining these arguments with
Lemma \ref{lemma: fact patterns: G(x,T) with fact pat lambda} we
obtain the following result.
\begin{lemma}
\label{lemma: fact patterns: sym pols and pattern lambda} A
polynomial $f:=T^n+a_{n-1}T^{n-1}\plp a_0\in\mathcal{P}$ has
factorization pattern $\bfs \lambda$ if and only if there exists
$\bfs{x}\in\fq^n$ of type $\bfs \lambda$ such that
\begin{equation}\label{eq: fact patterns: sym pols and pattern lambda}
a_k= (-1)^{{n-k}}\,\Pi_{n-k}(\bfs Y(\bfs x)) \quad(0\le k\le n-1).
\end{equation}
In particular, if $f$ is square--free, then there are $w(\bfs
\lambda)$ elements $\bfs x$ for which (\ref{eq: fact patterns: sym
pols and pattern lambda}) holds.
\end{lemma}

An easy consequence of this result is that we may express the
condition that an element of $\mathcal{A}:=\mathcal{A}(\bfs
L,\bfs\alpha)$ has factorization pattern $\bfs \lambda$ in terms of
the elementary symmetric polynomials $\Pi_1\klk\Pi_{n-r}$ of
$\fq[\bfs Y]$.
\begin{corollary}\label{coro: fact patterns: systems pattern lambda}
A polynomial $f:=T^n+a_{n-1}T^{n-1}\plp a_0\in\mathcal{A}$ has
factorization pattern $\bfs \lambda$ if and only if there exists
$\bfs{x}\in\fq^n$ of type $\bfs \lambda$ such that
\begin{equation}\label{eq: fact patterns: systems pattern lambda}
L_j\big((-1)^{n-r}\, \Pi_{n-r}(\bfs Y(\bfs x))\klk -\Pi_1(\bfs
Y(\bfs x))\big)+\alpha_j=0 \quad(1\le j\le m).
\end{equation}
In particular, if $f\in\mathcal{A}_{\bfs\lambda}$ is square--free,
then there are $w(\bfs \lambda)$ elements $\bfs x$ for which
(\ref{eq: fact patterns: systems pattern lambda}) holds.
\end{corollary}
%
%---------------------------------------------------------------------
%---------------------------------------------------------------------
%---------------------------------------------------------------------
%---------------------------------------------------------------------
%---------------------------------------------------------------------
%---------------------------------------------------------------------
%---------------------------------------------------------------------
%---------------------------------------------------------------------
%
\section{The geometry of the set of zeros of $R_1\klk R_m$}
\label{section: geometry of V}
Let $m$, $n$ and $r$ be positive integers with $q>n$ and $3\le r\le
n-m$. Given a factorization pattern $\bfs{\lambda}:=1^{\lambda_1}
\cdots n^{\lambda_n}$, consider the family $\mathcal{A}_{\bfs
\lambda}\subset\fq[T]$ of monic polynomials of degree $n$ having
factorization pattern $\bfs{\lambda}$, where $\mathcal{A}\subset
\mathcal{P}$ is the linear family defined in (\ref{eq: intro:
definition A}). In Corollary \ref{coro: fact patterns: systems
pattern lambda} we associate to $\mathcal{A}_{\bfs \lambda}$ the
following polynomials of $\fq[\bfs X]:=\fq[X_1,\ldots, X_n]$:
\begin{equation}\label{eq: geometry: def R_j}
R_j:=R_j^{\bfs\lambda}:= L_j\big((-1)^{n-r}\, \Pi_{n-r}(\bfs Y(\bfs
X))\klk -\Pi_1(\bfs Y(\bfs X))\big)+\alpha_j\quad (1\le j\le m).
\end{equation}
The set of common $\fq$--rational zeros of $R_1\klk R_m$ are
relevant for our purposes.

Up to the linear change of coordinates defined by $\bfs Y:=(Y_1\klk
Y_n)$, we may express each $R_j$ as a linear polynomial in the first
$n-r$ elementary symmetric polynomials $\Pi_1,\ldots,\Pi_{n-r}$ of
$\fq[\bfs Y]$. More precisely, let $Z_1\klk Z_{n-r}$ be new
indeterminates over $\cfq$. Then we have that
$$R_j=S_j(\Pi_1,\ldots,\Pi_{n-r})\quad (1\le j\le m),$$
where $S_1\klk S_m\in \fq[Z_1,\ldots,Z_{n-r}]$ are elements of
degree $1$ whose homogeneous components of degree 1 are linearly
independent in $\cfq[Z_1,\ldots,Z_{n-r}]$, namely the Jacobian
matrix $(\partial \bfs{S}/\partial \bfs{Z})$ of $S_1\klk S_m$ with
respect to $\bfs{Z} :=(Z_1,\ldots,Z_{n-r})$ has full rank $m$.

In this section we obtain critical information on the geometry of
the set of common zeros of the polynomials $R_1\klk R_m$ that will
allow us to establish estimates on their number of common
$\fq$--rational zeros.
%
% ----------------------------------------------------------------
% ----------------------------------------------------------------
% ----------------------------------------------------------------
% ----------------------------------------------------------------
%
\subsection{Notions of algebraic geometry}
Since our approach relies on tools of algebraic geometry, we briefly
collect the basic definitions and facts that we need in the sequel.
We use standard notions and notations of algebraic geometry, which
can be found in, e.g., \cite{Kunz85,Shafarevich94}.

We denote by $\A^n$ the affine $n$--dimensional space $\cfq{\!}^{n}$
and by $\Pp^n$ the projective $n$--dimensional space over
$\cfq{\!}^{n+1}$. Both spaces are endowed with their respective
Zariski topologies, for which a closed set is the zero locus of
polynomials of $\cfq[X_1,\ldots, X_{n}]$ or of homogeneous
polynomials of  $\cfq[X_0,\ldots, X_{n}]$. For $\K:=\fq$ or
$\K:=\cfq$, we say that a subset $V\subset \A^n$ is an {\sf affine
$\K$--definable variety} (or simply {\sf affine $\K$--variety}) if
it is the set of common zeros in $\A^n$ of polynomials $F_1,\ldots,
F_{m} \in \K[X_1,\ldots, X_{n}]$. Correspondingly, a {\sf projective
$\K$--variety} is the set of common zeros in $\Pp^n$ of a family of
homogeneous polynomials $F_1,\ldots, F_m \in\K[X_0 ,\ldots, X_n]$.
We shall frequently denote by $V(F_1\klk F_m)$ the affine or
projective $\K$--variety consisting of the common zeros of
polynomials $F_1\klk F_m$. The set $V(\fq):=V\cap \fq^n$ is the set
of {\sf $\fq$--rational} {\sf points} of $V$.

A $\K$--variety $V$ is $\K$--{\sf irreducible} if it cannot be
expressed as a finite union of proper $\K$--subvarieties of $V$.
Further, $V$ is {\sf absolutely irreducible} if it is
$\cfq$--irreducible as a $\cfq$--variety. Any $\K$--variety $V$ can
be expressed as an irredundant union $V=\mathcal{C}_1\cup
\cdots\cup\mathcal{C}_s$ of irreducible (absolutely irreducible)
$\K$--varieties, unique up to reordering, which are called the {\sf
irreducible} ({\sf absolutely irreducible}) $\K$--{\sf components}
of $V$.

%The set $V(\fq):=V\cap \fq^n$ is the set of {\sf $q$--rational
%points} of $V$. Studying the number of elements of $V(\fq)$ is a
%classical problem. The existence of $q$--rational points depends
%upon many circumstances concerning the geometry of the underlying
%variety.

For a $\K$-variety $V$ contained in $\A^n$ or $\Pp^n$, we denote by
$I(V)$ its {\sf defining ideal}, namely the set of polynomials of
$\K[X_1,\ldots, X_n]$, or of $\K[X_0,\ldots, X_n]$, vanishing on
$V$. The {\sf coordinate ring} $\K[V]$ of $V$ is defined as the
quotient ring $\K[X_1,\ldots,X_n]/I(V)$ or
$\K[X_0,\ldots,X_n]/I(V)$. The {\sf dimension} $\dim V$ of a
$\K$-variety $V$ is the length $r$ of the longest chain
$V_0\varsubsetneq V_1 \varsubsetneq\cdots \varsubsetneq V_r$ of
nonempty irreducible $\K$-varieties contained in $V$. A
$\K$--variety $V$ is called {\sf equidimensional} if all the
irreducible $\K$--components of $V$ are of the same dimension.

The {\sf degree} $\deg V$ of an irreducible $\K$-variety $V$ is the
maximum number of points lying in the intersection of $V$ with a
linear space $L$ of codimension $\dim V$, for which $V\cap L$ is a
finite set. More generally, following \cite{Heintz83} (see also
\cite{Fulton84}), if $V=\mathcal{C}_1\cup\cdots\cup \mathcal{C}_s$
is the decomposition of $V$ into irreducible $\K$--components, we
define the degree of $V$ as
$$\deg V:=\sum_{i=1}^s\deg \mathcal{C}_i.$$
With this definition of degree, we have the following {\em B\'ezout
inequality} (see \cite{Heintz83,Fulton84,Vogel84}): if $V$ and $W$
are $\K$--varieties, then
\begin{equation}\label{eq: geometry: Bezout}
\deg (V\cap W)\le \deg V \cdot \deg W.
\end{equation}

%We shall also make use of the following well--known identities
%relating the degree of an affine $\K$--variety $V \subset \A^n$, the
%degree of its projective closure (with respect to the projective
%Zariski $\K$--topology) $\overline{V} \subset \Pp^n$ and the degree
%of the affine cone $\widetilde{V}$ of $\overline{V}$ (see, e.g.,
%\cite[Proposition 1.11]{CaGaHe91}):
%%
%$$\deg V= \deg \overline{V}= \deg\tilde{V}.$$

Let $V$ and $W$ be irreducible affine $\K$--varieties of the same
dimension and let $f:V\to W$ be a regular map for which
$\overline{f(V)}=W$ holds, where $\overline{f(V)}$ denotes the
closure of $f(V)$ with respect to the Zariski topology of $W$. Such
a map is called {\sf dominant}. Then $f$ induces a ring extension
$\K[W]\hookrightarrow \K[V]$ by composition with $f$. We say that
the dominant map $f$ is a {\sf finite morphism} if this extension is
integral, namely if each element $\eta\in\K[V]$ satisfies a monic
equation with coefficients in $\K[W]$. A basic fact is that a
dominant finite morphism is necessarily closed. Another fact
concerning dominant finite morphisms we shall use in the sequel is
that the preimage $f^{-1}(S)$ of an irreducible closed subset
$S\subset W$ is equidimensional of dimension $\dim S$  (see, e.g.,
\cite[\S 4.2, Proposition]{Danilov94}).

Let $V\subset\A^n$ be a variety and let $I(V)\subset
\cfq[X_1,\ldots, X_n]$ be the defining ideal of $V$. Let $\bfs{x}$
be a point of $V$. The {\sf dimension} $\dim_{\bfs{x}}V$ {\sf of}
$V$ {\sf at} $\bfs{x}$ is the maximum of the dimensions of the
irreducible components of $V$ that contain $\bfs{x}$. If
$I(V)=(F_1,\ldots, F_m)$, the {\sf tangent space}
$\mathcal{T}_{\bfs{x}}V$ to $V$ at $\bfs{x}$ is the kernel of the
Jacobian matrix $(\partial F_i/\partial X_j)_{1\le i\le m,1\le j\le
n}(\bfs{x})$ of $F_1,\ldots, F_m$ with respect to $X_1,\ldots, X_n$
at $\bfs{x}$. The point $\bfs{x}$ is {\sf regular} if
$\dim\mathcal{T}_{\bfs{x}}V=\dim_{\bfs{x}}V$ holds. Otherwise, the
point $\bfs{x}$ is called {\sf singular}. The set of singular points
of $V$ is the {\sf singular locus} $\mathrm{Sing}(V)$ of $V$. A
variety is called {\sf nonsingular} if its singular locus is empty.
For a projective variety, the concepts of tangent space, regular and
singular point can be defined by considering an affine neighborhood
of the point under consideration.

Elements $F_1 \klk F_r$ in $\cfq[X_1\klk X_n]$ or in $\cfq[X_0\klk
X_n]$ form a {\sf regular sequence} if $F_1$ is nonzero and each
$F_i$ is not a zero divisor in the quotient ring $\cfq[X_1\klk
X_n]/(F_1\klk F_{i-1})$ or $\cfq[X_0\klk X_n]/(F_1\klk F_{i-1})$ for
$2\le i\le r$. In such a case, the (affine or projective) variety
$V:=V(F_1\klk F_r)$ they define is equidimensional of dimension
$n-r$, and is called a {\sf set--theoretic} {\sf complete
intersection}. If, in addition, the ideal $(F_1\klk F_r)$ generated
by $F_1\klk F_r$ is radical, then $V$ is an {\sf ideal--theoretic}
{\sf complete intersection}. If $V\subset\Pp^n$ is an
ideal--theoretic complete intersection of dimension $n-r$, and $F_1
\klk F_r$ is a system of homogeneous generators of $I(V)$, the
degrees $d_1\klk d_r$ depend only on $V$ and not on the system of
generators. Arranging the $d_i$ in such a way that $d_1\geq d_2 \geq
\cdots \geq d_r$, we call $\bfs d:=(d_1\klk d_r)$ the {\sf
multidegree} of $V$. The so--called {\em B\'ezout theorem} (see,
e.g., \cite[Theorem 18.3]{Harris92}) asserts that
\begin{equation}\label{eq: geometry: Bezout eq}
\deg V=d_1\cdots d_r.
\end{equation}

In what follows we shall deal with a particular class of complete
intersections, which we now define. A variety is {\sf regular in
codimension $m$} if the singular locus $\mathrm{Sing}(V)$ of $V$ has
codimension at least $m+1$ in $V$, namely if $\dim V-\dim
\mathrm{Sing}(V)\ge m+1$. A complete intersection $V$ which is
regular in codimension 1 is called {\sf normal} (actually, normality
is a general notion that agrees on complete intersections with the
one defined here). A fundamental result for projective complete
intersections is the Hartshorne connectedness theorem (see, e.g.,
\cite[Theorem VI.4.2]{Kunz85}), which we now state. If
$V\subset\Pp^n$ is a set--theoretic complete intersection and
$W\subset V$ is any subvariety of codimension at least 2, then
$V\setminus W$ is connected in the Zariski topology of $\Pp^n$.
Applying the Hartshorne connectedness theorem with
$W:=\mathrm{Sing}(V)$, one deduces the following result.
\begin{theorem}\label{theorem: normal complete int implies irred}
If $V\subset\Pp^n$ is a normal set--theoretic complete intersection,
then $V$ is absolutely irreducible.
\end{theorem}
%
% ----------------------------------------------------------------
% ----------------------------------------------------------------
% ----------------------------------------------------------------
% ----------------------------------------------------------------
%
\subsection{The singular locus of the variety $V(R_1\klk R_m)$}
\label{subsec: geometry: singular locus symmetric complete inters}
With the notations and assumptions of the beginning of Section
\ref{section: geometry of V}, let $V:=V^{\bfs{\lambda}}\subset \A^n$
be the affine variety defined by the polynomials $R_1\klk
R_m\in\fq[\bfs X]$ of \eqref{eq: geometry: def R_j}. The main result
of this section asserts that $V$ is regular in codimension one. From
this result we will be able to conclude that $V$ is a normal
ideal--theoretic complete intersection.

In the sequel we shall frequently express the points of $\A^n$ in
the coordinate system $\bfs Y:=(Y_1\klk Y_n)$, where $Y_1\klk Y_n$
are the linear forms of (\ref{eq: fact patterns: def linear forms
Y}). Let $Z_1,\ldots, Z_n$ be new indeterminates over $\cfq$, set
$\bfs{Z} :=(Z_1,\ldots,Z_{n-r})$ and let $S_1\klk
S_m\in\fq[\bfs{Z}]$ be the linear polynomials for which
$R_j=S_j(\Pi_1\klk\Pi_{n-r})$ holds for $1\le j\le m$, where
$\Pi_1\klk\Pi_{n-r}$ are the first $n-r$ elementary symmetric
polynomials of $\fq[\bfs Y]$. Recall that, by hypothesis, the
Jacobian matrix $(\partial \bfs{S}/\partial \bfs{Z})$ of
$\bfs{S}:=(S_1\klk S_m)$ with respect to $\bfs{Z}$ has full rank
$m$.

We now consider $S_1\klk S_m$ as elements of $\fq[Z_1\klk Z_n]$.
Since the Jacobian matrix $(\partial \bfs{S}/\partial \bfs{Z})$ has
full rank $m$, the linear affine variety $W\subset\A^n$ that
$S_1\klk S_m$ define has dimension $n-m$. Consider the following
surjective mapping:
\begin{align*}
   \bfs{\Pi}^{\bfs{n}}: \A^n& \rightarrow  \A^n
   \\
   \bfs y & \mapsto  (\Pi_1(\bfs y),\ldots,\Pi_n(\bfs y)).
\end{align*}
It is easy to see that $\bfs\Pi^{\bfs{n}}$ is a dominant finite
morphism (see, e.g., \cite[\S 5.3, Example 1]{Shafarevich94}). In
particular, the preimage $(\bfs\Pi^{\bfs{n}})^{-1}(\mathcal{Z})$ of
an irreducible affine variety $\mathcal{Z}\subset\A^n$ of dimension
$m$ is equidimensional and of dimension $m$.

Observe that the affine linear variety $W_j:=V(S_1\klk
S_j)\subset\A^n$ is equidimensional of dimension $n-j$. This implies
that the affine variety $(\bfs\Pi^{\bfs{n}})^{-1}(W_j)=V(R_1\klk
R_j)\subset\A^n$ is equidimensional of dimension $n-j$. We conclude
that $R_1\klk R_m$ form a regular sequence of $\fq[\bfs Y]$ and
deduce the following result.
\begin{lemma}\label{lemma: geometry: V is complete inters}
Let $V\subset \A^n$ be the affine variety defined by $R_1\klk R_m$.
Then $V$ is a set--theoretic complete intersection of dimension
$n-m$.
\end{lemma}

%Let $(\partial \bfs{S}/\partial \bfs{Z}):=(\partial S_j/\partial
%Z_k)_{d-s\le j\le r-1,1\le k\le s}$ be the Jacobian matrix of
%$S_{d-s}\klk S_{r-1}$ with respect to $Z_1\klk Z_s$. Our assumptions
%on $s$, $d$ and $r$ imply $r-d+s\le s$ and thus, $(\partial
%\bfs{S}/\partial \bfs{Z})$ has full rank if and only if
%$\mathrm{rank}(\partial \bfs{S}/\partial \bfs{Z})=r-d+s$ holds.
%
%From (H1) and (H2) we immediately conclude that Furthermore, as a
%consequence of \cite[Theorem 18.15]{Eisenbud95} we conclude that
%$S_{d-s}\klk S_{r-1}$ define a radical ideal, and hence $W_r$ is an
%ideal--theoretic complete intersection.
%

Next we analyze the dimension of the singular locus of $V$. Assume
without loss of generality that $(\partial \bfs{S}/\partial
\bfs{Z})$ is lower triangular in row--echelon form. Let $1\le
i_1<\cdots< i_m\le n-r$ be the indices corresponding to the pivots.
Let $\mathcal{I}:=\{i_1\klk i_m\}$ and $\mathcal{J}:=\{j_1\klk
j_{n-r-m}\}:=\{1\klk n-r\}\setminus \mathcal{I}$. Then the Jacobian
matrix
\begin{equation}\label{eq: geometry: extended Jacobian S}
\mathcal{M}:=\big(\partial (S_1\klk S_m,Z_{j_1}\klk Z_{j_{n-r-m}})/
\partial \bfs{Z}\big)\in\cfq^{(n-r)\times(n-r)}
\end{equation}
is invertible. Let $B_0\klk B_{n-m-1}$ be new indeterminates over
$\cfq$ and define $S_{m+k}:= Z_{j_k}+B_{n-m-k}$ $(1\le k\le n-r-m)$
and $S_k:=Z_k+B_{n-k}$ $(n-r+1\le k\le n)$. Set
$\bfs{B}:=(B_{n-m-1}\klk B_0)$, $\bfs{S}^e:=(S_1\klk S_n)$ and
$\bfs{Z}^e:=(Z_1\klk Z_n)$. Observe that the Jacobian matrix
$$\big(\partial \bfs{S}^e/\partial \bfs Z^e\big)=
\big(\partial (S_1\klk S_m,Z_{j_1}\klk Z_{j_{n-r-m}},Z_{n-r+1}\klk
Z_n)/
\partial \bfs Z^e\big)$$
is also invertible. Consider the following surjective morphism of
affine varieties:
 \begin{align*}
 \bfs  \Pi : \A^n & \rightarrow  \A^{n-r}
   \\
   \bfs y & \mapsto  (\Pi_1(\bfs y),\ldots,\Pi_{n-r}(\bfs y)).
 \end{align*}
Finally, we introduce the affine variety $V^e\subset\A^{2n-m}$
defined in the following way:
$$V^e:=\{(\bfs y,\bfs{b})\in\A^n\times\A^{n-m}:
S_j(\bfs\Pi(\bfs y),\bfs{b})=0\ (1\le j\le n)\}.$$

In order to establish a relation between $V$ and $V^e$, let $(\bfs
y,\bfs{b})$ be an arbitrary point of $V^e$. Then $S_j(\bfs\Pi(\bfs
y),\bfs{b})=S_j(\bfs \Pi(\bfs y))=0$ holds for $1\le j\le m$, which
implies that $\bfs y\in V$. This shows the following regular mapping
of affine varieties is well--defined:
 \begin{align*}
   \Phi_1^e : V^e & \rightarrow  V
   \\
   (\bfs y,\bfs{b})& \mapsto  \bfs y.
 \end{align*}
Furthermore, by the definition of $V^e$ it is easy to see that
$\Phi_1^e$ is an isomorphism of affine varieties, whose inverse is
the following mapping:
 \begin{align*}
   \Psi^e : V & \rightarrow  V^e
   \\
   \bfs y & \mapsto  \big(\bfs y,-\Pi_{j_1}(\bfs y),\dots,
   -\Pi_{j_{n-r-m}}(\bfs y),-\Pi_{n-r+1}(\bfs y)\klk -\Pi_n(\bfs y)\big).
 \end{align*}
We conclude that $V^e$ is an affine equidimensional variety of
dimension $n-m$.

Our aim is to show that the singular locus $\Sigma$ of $V$ has
codimension at least 2 in $V$. For this purpose, we shall show that
the singular $\Sigma^e$ of $V^e$ has codimension at least 2 in
$V^e$.

Let $R_{m+k}:=S_{m+k}(\bfs\Pi,\bfs{B})$ for $1\le k\le n-m$. We
denote by $(\partial \bfs{R}/\partial \bfs Y)$ the Jacobian matrix
of $\bfs{R}:=(R_1\klk R_m)$ with respect to $\bfs Y$ and by
$(\partial \bfs{R}^e/\partial (\bfs Y,\bfs{B}))$ the Jacobian matrix
of $\bfs{R}^e:=(R_1\klk R_n)$ with respect to $\bfs Y$ and
$\bfs{B}$. The relation between the singular locus of $V$ and $V^e$
is expressed in the following remark.
\begin{remark}\label{remark: geometry: jacobians of full rank}
For $\bfs y\in V$, let $(\bfs y,\bfs{b}):=\Psi^e(\bfs y)$. Then
$(\partial \bfs{R}/\partial \bfs Y)(\bfs y)$ is of full rank $m$ if
and only if $\big(\partial \bfs{R}^e/(\partial \bfs
Y,\bfs{B})\big)(\bfs y,\bfs{b})$ is of full rank $n$.
\end{remark}
\begin{proof}
Let $\bfs y\in V$ be a point as in the statement of the remark.
%By the fact that $\Pi_1^e$ is an isomorphism, we conclude that there
%exists a unique $\bfs{b}\in\A^{n-m}$ with $(\bfs{x},\bfs{b})\in
%V^e$, namely $(\bfs{x},\bfs{b}):=\Pi_1^e(\bfs{x})$. We claim that
%$\bfs{b}$ satisfies the conclusion in the statement of the lemma.
By the definition of $\bfs{R}^e$ it follows that $\big(\partial
\bfs{R}^e/(\partial \bfs Y,\bfs{B})\big)(\bfs y,\bfs{b})$ has a
block structure as follows:
$$\frac{\partial
\bfs{R}^e}{\partial (\bfs Y,\bfs{B})}(\bfs y,\bfs{b})= \left(
  \begin{array}{cc}
   \frac{\partial \bfs{R}}{\partial \bfs Y}(\bfs y) & \bfs{0} \\[1ex]
    \frac{\partial (\bfs{R}^e\setminus\bfs R)}{\partial \bfs Y}(\bfs y,\bfs{b})
    & \frac{\partial (\bfs{R}^e\setminus \bfs R)}{\partial \bfs{B}}(\bfs y,\bfs{b}) \\
  \end{array}
\right)= \left(
  \begin{array}{cc}
   \frac{\partial \bfs{R}}{\partial \bfs Y}(\bfs y) & \bfs{0} \\[1ex]
    \frac{\partial (\bfs{R}^e\setminus\bfs R)}{\partial \bfs Y}(\bfs y,\bfs{b}) & \bfs{I} \\
  \end{array}
\right),
$$
where $\bfs{0}$ denotes a zero $m\times (n-m)$--matrix and $\bfs{I}$
denotes an $(n-m)\times(n-m)$--identity matrix. The conclusion of
the remark readily follows.
\end{proof}

In order to obtain an upper bound on the dimension of the singular
locus of $V^e$, we consider the following projection mapping:
 \begin{align*}
   \Phi_2^e : V^e & \rightarrow  \A^{n-m}
   \\
   (\bfs y,\bfs{b})& \mapsto  \bfs{b}.
 \end{align*}
We shall analyze the image under $\Phi_2^e$ of the singular locus of
$V^e$. The following result will allow us to draw conclusions
concerning the singular locus of $V^e$ from the analysis of its
image under $\Phi_2^e$.
\begin{lemma}\label{lemma: geometry: pi_2 is finite}
$\Phi_2^e$ is a dominant finite morphism.
\end{lemma}
\begin{proof}
Let $\alpha_i:=S_i(\bfs{0})$ for $1\le i\le m$ and let
$\bfs{\alpha}:=(\alpha_1\klk \alpha_m)$. Set
$\bfs{B}_J:=(B_{n-m-1}\klk B_r)$. By the definition of the matrix
$\mathcal{M}$ of (\ref{eq: geometry: extended Jacobian S}) it
follows that
$$(R_1\klk R_{n-r})=\mathcal{M}\cdot \bfs\Pi+
(\bfs\alpha,\bfs{B}_J)^t,$$
where $\bfs\Pi:=(\Pi_1\klk \Pi_{n-r})$. As a consequence, we have
$\bfs\Pi(\bfs y)+\mathcal{M}^{-1} (\bfs\alpha,\bfs{b}_J)^t=\bfs{0}$
for any $(\bfs y,\bfs{b})\in V^e$. Denote by $\bfs{m}_j$ the $j$th
row of $\mathcal{M}^{-1}$ for $1\le j\le n-r$. Then
$$\Pi_j(\bfs y)+\bfs{m}_j\cdot (\bfs\alpha,\bfs{b}_J)^t=0\quad (1\le j\le n-r)$$
for any $(\bfs y,\bfs{b})\in V^e$. Furthermore, the identity
$$(y_j)^n+\sum_{k=1}^n(-1)^k\,\Pi_k(\bfs y)\,(y_j)^{n-k}=0$$
holds for $1\le j\le n$. Combining the two previous identities and
the definition of $\bfs{S}^e$, we conclude that the polynomial
$$P_j:=Y_j^n-\sum_{k=1}^{n-r}(-1)^k\,\bfs{m}_k\cdot (\bfs\alpha,
\bfs{B}_J)^tY_j^{n-k} -\sum_{k=n-r+1}^n(-1)^kB_{n-k}Y_j^{n-k}$$
vanishes identically over $V^e$ for $1\le j\le n$.

Let $\bfs{b}\in\A^{n-m}$ be a point of the image of $\Phi_2^e$ and
let $\bfs y\in V$ be an arbitrary point with $(\bfs y,\bfs{b})\in
V^e$. As the identity $P_j(y_j,\bfs{b})=0$ holds for $1\le j\le n$,
the fiber of $\bfs{b}$ under $\Phi_2^e$ has dimension zero. Then the
theorem on the dimension of fibers (see, e.g., \cite[\S I.6.3,
Theorem 7]{Shafarevich94}) asserts that $\dim V^e -\dim
\Phi_2^e(V^e)\le \dim(\Phi_2^e)^{-1}(\bfs{b})= 0$, namely $\dim
\Phi_2^e(V^e)\ge n-m$. It follows that $\Phi_2^e$ is dominant.

Furthermore, since $P_j(Y_j,\bfs{B})=0$ holds in $\cfq[V^e]$ for
$1\le j\le n$, we see that the ring extension
$\cfq[\bfs{B}]\hookrightarrow \cfq[V^e]$ is integral. This implies
that $\Phi_2^e$ is a finite morphism and finishes the proof of the
lemma.
\end{proof}

Next we obtain a partial characterization of the singular locus of
$V^e$. Here we use the fact that $V$ is defined by symmetric
polynomials.
\begin{lemma}\label{lemma: geometry: coordinates sing point}
Let $(\bfs y,\bfs{b})\in V^e$ be a point for which $(\partial
\bfs{R}^e/\partial (\bfs Y,\bfs{B}))(\bfs y,\bfs{b})$ has not full
rank. Then there exist $1\le i<j<k<l\le n$ such that $y_i=y_j$ and
$y_k=y_l$.
\end{lemma}
\begin{proof}
Let $(\bfs y,\bfs{b})\in V^e$ be a point as in the statement of the
lemma. According to Remark \ref{remark: geometry: jacobians of full
rank}, the Jacobian matrix  $(\partial \bfs{R}/\partial \bfs Y)(\bfs
y)$ is not of full rank. Since $\bfs{R}=\bfs{S}\circ\bfs\Pi$, by the
chain rule we obtain
$$
\left(\frac{\partial \bfs{R}}{\partial \bfs
Y}\right)=\left(\frac{\partial \bfs{S}}{\partial
\bfs{Z}}\circ\bfs\Pi\right)\cdot \left(\frac{\partial\bfs
\Pi}{\partial \bfs Y}\right).
$$

Let $\bfs{v}\in\A^m$ a nonzero vector in the left kernel of
$(\partial \bfs{R}/\partial \bfs Y)(\bfs y)$. Then
$$\mathbf{0}=
\bfs{v}\cdot \left(\frac{\partial \bfs{R}}{\partial \bfs
Y}\right)(\bfs y)=\bfs{v}\cdot\left(\frac{\partial \bfs{S}}{\partial
\bfs{Z}}\right)\big(\bfs\Pi(\bfs y)\big)\cdot \left(\frac{\partial
\bfs\Pi}{\partial \bfs Y}\right)(\bfs y).$$
Since the Jacobian matrix $(\partial \bfs{S}/\partial
\bfs{Z})\big(\bfs\Pi(\bfs y)\big)$ has full rank, we deduce that the
vector $\bfs{w}:=\bfs{v}\cdot\left({\partial \bfs{S}}/{\partial
\bfs{Z}}\right)\big(\bfs\Pi(\bfs y)\big)\in\A^{n-r}$ is nonzero and
$$\bfs{w}\cdot \left(\frac{\partial\bfs \Pi}{\partial \bfs Y}\right)(\bfs y)= \mathbf{0}.$$
Hence, all the maximal minors of $(\partial \bfs\Pi/\partial \bfs
Y)(\bfs y)$ must vanish.

Arguing as in the proof of \cite[Theorem 3.2]{CeMaPePr14} (see also
\cite[Theorem 3.1]{CaMaPr12}), we conclude that $\bfs y$ has at most
$n-4$ pairwise--distinct coordinates. In particular, there exist
$1\le i<j\le n-2$ with $y_i= y_j$. Assume without loss of generality
that $i=1$ and $j=2$. Then there exist $3\le k<l\le n$ with
$y_k=y_l$. This finishes the proof of the lemma. \end{proof}

Now we obtain an upper bound on the dimension of the singular locus
of $V^e$.
\begin{proposition}
\label{prop: geometry: set of points x,b with sing Jacobian} Let
$p>2$. The set of points $(\bfs y,\bfs{b})\in V^e$ for which the
Jacobian matrix $(\partial \bfs{R}^e/\partial (\bfs Y,\bfs{B}))(\bfs
y,\bfs{b})$ has not full rank, has codimension at least $2$ in
$V^e$. In particular, the singular locus of $V^e$ has codimension at
least $2$ in $V^e$.
\end{proposition}
\begin{proof}
We use the notations of the proof of Lemma \ref{lemma: geometry:
pi_2 is finite}. In the proof of Lemma \ref{lemma: geometry: pi_2 is
finite} we show that the polynomial
$$P_j(Y_j,\bfs{B}):=
Y_j^n-\sum_{k=1}^{n-r}(-1)^k\,\bfs{m}_k\cdot
(\bfs\alpha,\bfs{B}_J)^tY_j^{n-k}
-\sum_{k=n-r+1}^n(-1)^kB_{n-k}Y_j^{n-k}$$
vanishes identically on $V^e$ for $1\le j\le n$. Let $(\bfs
y,\bfs{b})\in V^e$ be a point as in the statement of the proposition
and let $P_{\bfs{b}}\in\cfq[T]$ be the polynomial
\begin{align*}
P_{\bfs{b}}&:=  T^n-\sum_{k=1}^{n-r}(-1)^k\,\bfs{m}_k\cdot
(\bfs\alpha,\bfs{b}_J)^tT^{n-k} -\sum_{k=n-r+1}^n(-1)^kb_{n-k}T^{n-k}\\
&=T^n+\sum_{k=1}^n(-1)^k\,\Pi_k(\bfs y)\,T^{n-k}=
\prod_{j=1}^n(T-y_j).
\end{align*}
Since the roots of $P_{\bfs{b}}$ in $\cfq$ are the coordinates of
$\bfs y$, by Lemma \ref{lemma: geometry: coordinates sing point} we
have that $P_{\bfs{b}}$ has either two distinct multiple roots, or a
root of multiplicity at least three.

On one hand, \cite[Lemma 4.5]{MaPePr13} shows that the set
consisting of the elements $\bfs{b}\in\A^{n-m}$ such that
$P_{\bfs{b}}$ has two distinct multiple roots is contained in a
subvariety of codimension 2 of $\A^{n-m}$. On the other hand,
\cite[Lemma 4.7]{MaPePr13} proves that the subset of $\A^{n-m}$
formed by the elements $\bfs{b}$ for which $P_{\bfs{b}}$ has a root
of multiplicity at least three is contained in a subvariety of
codimension 2 of $\A^{n-m}$. As a consequence, the image under
$\Phi_2^e$ of the set of points $(\bfs y,\bfs{b})\in V^e$ as in the
statement of the lemma is contained in a subvariety of codimension 2
of $\A^{n-m}$.

Lemma \ref{lemma: geometry: pi_2 is finite} asserts that $\Phi_2^e$
is dominant finite morphism. Therefore, as the inverse image of a
codimension--2 subvariety of $\A^{n-m}$ is a codimension--2
subvariety of $V^e$, the first assertion of the proposition is
deduced.

Now we consider the second assertion of the proposition. Let $(\bfs
y,\bfs{b})$ be a singular point of $V^e$ and let $\mathcal{T}_{\bfs
y}V^e$ be the tangent space of $V^e$ at $\bfs y$. Since
$V^e=V(R_1\klk R_n)$, for any $\bfs{v}\in \mathcal{T}_{\bfs y}V^e$
we have $(\partial \bfs{R}^e/\partial \bfs Y)(\bfs
y,\bfs{b})\cdot\bfs{v}=\bfs{0}$. If the Jacobian matrix $(\partial
\bfs{R}^e/\partial \bfs Y)(\bfs y,\bfs{b})$ had full rank, then
$\mathcal{T}_{\bfs y}V^e$ would have dimension at most $n-m$,
contradicting thus our assumption on $(\bfs y,\bfs{b})$. Hence, the
second assertion readily follows.
\end{proof}

Finally, we are able to establish our main result concerning the
dimension of the singular locus of $V$.
\begin{theorem}\label{theorem: geometry: dimension singular locus}
Let $p>2$. The set of points $\bfs y\in V$ for which $(\partial
\bfs{R}/\partial \bfs Y)(\bfs y)$ has not full rank, has codimension
at least $2$ in $V$. In particular, the singular locus $\Sigma$ of
$V$ has codimension at least $2$ in $V$.
\end{theorem}
\begin{proof}
Recall that the projection mapping $\Phi_1^e : V^e \rightarrow  V$
defined by $\Phi_1^e(\bfs y,\bfs{b}):=\bfs y$ is an isomorphism of
affine varieties. Furthermore, Remark \ref{remark: geometry:
jacobians of full rank} asserts that the image under $\Phi_1^e$ of
the set of points $(\bfs y,\bfs{b})\in V^e$ for which $(\partial
\bfs{R}^e/\partial (\bfs Y,\bfs{B}))(\bfs y,\bfs{b})$ has not full
rank is the set of points $\bfs y\in V$ as in the statement of the
theorem. Proposition \ref{prop: geometry: set of points x,b with
sing Jacobian} shows that the former is contained in a
codimension--2 subvariety of $V^e$, which implies that the latter is
contained in a codimension--2 subvariety of $V$. This proves the
first assertion of the theorem.

Now let $\bfs y$ be an arbitrary point $\Sigma$. By Lemma
\ref{lemma: geometry: V is complete inters} we have $\dim
\mathcal{T}_{\bfs y}V>n-m$. This implies that
$\mathrm{rank}\left({\partial \bfs{R}}/{\partial \bfs Y}\right)(\bfs
y)<m$, for otherwise we would have $\dim \mathcal{T}_{\bfs y}V\le
n-m$, contradicting thus the fact that $\bfs y$ is a singular point
of $V$. From the first assertion, already proved, we easily deduce
the second assertion of the theorem.
\end{proof}
From Lemma \ref{lemma: geometry: V is complete inters} and Theorem
\ref{theorem: geometry: dimension singular locus} we obtain further
algebraic and geometric consequences concerning the polynomials
$R_1\klk R_m$ and the variety $V$. By Theorem \ref{theorem:
geometry: dimension singular locus} we have that the set of points
$\bfs y\in V$ for which the Jacobian matrix $(\partial
\bfs{R}/\partial \bfs Y)(\bfs y)$ has not full rank has codimension
at least $2$ in $V$. Since $R_1\klk R_m$ form a regular sequence of
$\fq[\bfs Y]$, from \cite[Theorem 18.15]{Eisenbud95} we conclude
that $R_1\klk R_m$ define a radical ideal of $\fq[\bfs Y]$.

On the other hand, recall that the matrix $(\partial
\bfs{S}/\partial \bfs{Z})$ was supposed to be lower triangular in
row--echelon form, the indices $i_1\klk i_m$ corresponding to the
positions of the pivots of $(\partial \bfs{S}/\partial \bfs{Z})$.
Then each polynomial $R_j$ has degree $i_j$ for $1\le j\le m$. By
the B\'ezout inequality \eqref{eq: geometry: Bezout} we have $\deg
V\le \prod_{j=1}^m\deg R_j=i_1\cdots i_m$. In other words, we have
the following statement.
\begin{corollary}\label{coro: geometry: radicality and degree Vr}
Let $p>2$. The polynomials $R_1\klk R_m$ define a radical ideal and
the variety $V$ has degree $\deg V\le\prod_{j=1}^m\deg R_j=i_1\cdots
i_m$.
\end{corollary}
%
% ----------------------------------------------------------------
% ----------------------------------------------------------------
% ----------------------------------------------------------------
% ----------------------------------------------------------------
%
\subsection{The projective closure of $V$}
In order to obtain estimates on the number of $\fq$--rational points
of $V$ we need information concerning the behavior of $V$ ``at
infinity''. For this purpose, we consider the projective closure
$\mathrm{pcl}(V)\subset\mathbb{P}^n$ of $V$, whose definition we now
recall. Consider the embedding of $\A^n$ into the projective space
$\Pp^n$ which assigns to any $\bfs y:=(y_1,\dots,y_n)\in\A^n$ the
point $(1:y_1:\dots:y_n)\in\Pp^n$. The closure
$\mathrm{pcl}(V)\subset\Pp^n$ of the image of $V$ under this
embedding in the Zariski topology of $\Pp^n$ is called the {\sf
projective closure} of $V$. The points of $\mathrm{pcl}(V)$ lying in
the hyperplane $\{Y_0=0\}$ are called the points of
$\mathrm{pcl}(V)$ at infinity.

It is well--known that $\mathrm{pcl} (V)$ is the variety of
$\mathbb{P}^n$ defined by the homogenization
$F^h\in\fq[Y_0,\ldots,Y_n]$ of each polynomial $F$ belonging to the
ideal $(R_1\klk R_m)\subset\fq[Y_1,\ldots,Y_n]$ (see, e.g., \cite[\S
I.5, Exercise 6]{Kunz85}). Denote by $(R_1\klk R_m)^h$ the ideal
generated by all the polynomials $F^h$ with $F\in (R_1\klk R_m)$.
Since the ideal $(R_1\klk R_m)$ is radical, the ideal $(R_1\klk
R_m)^h$ is also radical (see, e.g., \cite[\S I.5, Exercise
6]{Kunz85}). Furthermore, $\mathrm{pcl}(V)$ is equidimensional of
dimension $n-m$ (see, e.g., \cite[Propositions I.5.17 and
II.4.1]{Kunz85}) and degree equal to $\deg V$ (see, e.g.,
\cite[Proposition 1.11]{CaGaHe91}).

Now we discuss the behavior of $\mathrm{pcl} (V)$ at infinity.
According to \eqref{eq: geometry: def R_j}, each $R_j$ can be
expressed as
$$R_j=S_j(\Pi_1,\ldots,\Pi_{n-r})\quad (1\le j\le m),$$
where $S_1\klk S_m\in \fq[Z_1,\ldots,Z_{n-r}]$ are elements of
degree $1$ whose Jacobian matrix $(\partial \bfs{S}/\partial
\bfs{Z})$ with respect to $\bfs{Z} :=(Z_1,\ldots,Z_{n-r})$ has full
rank $m$. As before, we assume that $(\partial \bfs{S}/\partial
\bfs{Z})$ is lower triangular in row--echelon form, namely there
exist $1\le i_1<i_2<\cdots <i_m\le n-r$ such that
$$R_j=\alpha_j+\sum_{k=1}^{i_j}c_{j,k}\,\Pi_k,$$
where $c_{j,i_j}\not=0$ for $1\le j\le m$. Hence, the homogenization
of each $R_j$ is the following polynomial of $\fq[Y_0,\dots,Y_n]$:
\begin{equation} \label{eq: geometry: homogenizacion de R_j}
R_j^h=\alpha_jY_0^{i_j}+\sum_{k=1}^{i_j}c_{j,k}\,\Pi_k\,Y_0^{i_j-k}.
\end{equation}
It follows that $R_j^h(0,Y_1\klk Y_n)=\Pi_{i_j}$ ($1\le j\le m$).
Observe that the polynomials $\Pi_{i_1}\klk\Pi_{i_m}$ are a possible
choice for the polynomials $R_1\klk R_m$ of \eqref{eq: geometry: def
R_j}. Therefore, Lemma \ref{lemma: geometry: V is complete inters},
Theorem \ref{theorem: geometry: dimension singular locus} and
Corollary \ref{coro: geometry: radicality and degree Vr} hold with
$R_j:=\Pi_{i_j}$ for $1\le j\le m$.
\begin{lemma}\label{lemma: geometry: dim singular locus V at infinity}
Let $p>2$. Then $\mathrm{pcl}(V)$ has singular locus at infinity of
dimension at most $n-m-3$.
\end{lemma}
\begin{proof}
Let $\Sigma^{\infty}\subset\mathbb{P}^n$ denote the singular locus
of $\mathrm{pcl}(V)$ at infinity, namely the set of singular points
of $\mathrm{pcl}(V)$ lying in the hyperplane $\{Y_0=0\}$. Let $\bfs
y:=(0:y_1:\dots:y_n)$ be an arbitrary point of $\Sigma^{\infty}$.
Since the polynomials $R_j^h$ vanish identically in
$\mathrm{pcl}(V)$, we have $R_j^h(\bfs y)=\Pi_{i_j}(y_1\klk y_n)=0$
for $1\le j\le m$. Let $(\partial \bfs\Pi_\mathcal{I}/\partial \bfs
Y)$ be the Jacobian matrix of $\Pi_{i_1}\klk \Pi_{i_m}$ with respect
to $Y_1\klk Y_n$. We have
\begin{equation}\label{eq: geometry: rank singular points at infinity}
\mathrm{rank}\left(\frac{\partial \bfs\Pi_\mathcal{I}}{\partial \bfs
Y}\right)(\bfs y)<m,
\end{equation}
for if not, we would have $\dim \mathcal{T}_{\bfs
y}(\mathrm{pcl}(V))\le n-m$, which would imply that $\bfs y$ is a
nonsingular point of $\mathrm{pcl}(V)$, contradicting thus the
hypothesis on $\bfs y$.

Since the polynomials $\Pi_{i_1}\klk \Pi_{i_m}$ satisfy the
hypotheses of Theorem \ref{theorem: geometry: dimension singular
locus}, the points satisfying \eqref{eq: geometry: rank singular
points at infinity} form an affine equidimensional cone of dimension
at most $n-m-2$. We conclude that $\Sigma^\infty\subset\Pp^n$ has
dimension at most $n-m-3$.
\end{proof}
Now we are able to completely characterize the behavior of
$\mathrm{pcl}(V)$ at infinity.
\begin{theorem}\label{theorem: geometry: V at infinity}
Let $p>2$. Then $\mathrm{pcl}(V)\cap\{Y_0=0\}\subset\Pp^{n-1}$ is a
normal ideal--theoretic complete intersection of dimension $n-m-1$
and degree $i_1\cdots i_m$.
\end{theorem}
\begin{proof}
From \eqref{eq: geometry: homogenizacion de R_j} it is easy to see
that the polynomials $\Pi_{i_1}\klk \Pi_{i_m}$ vanish identically in
$\mathrm{pcl}(V)\cap\{Y_0=0\}$. Lemma \ref{lemma: geometry: V is
complete inters}, Theorem \ref{theorem: geometry: dimension singular
locus} and Corollary \ref{coro: geometry: radicality and degree Vr}
show that the variety of $\A^n$ defined by $\Pi_{i_1}\klk \Pi_{i_m}$
is an affine equidimensional cone of dimension $n-m$, degree at most
$i_1\cdots i_m$ and singular locus of dimension at most $n-m-2$. It
follows that the projective variety of $\Pp^{n-1}$ defined by these
polynomials is equidimensional of dimension $n-m-1$, degree at most
$i_1\cdots i_m$ and singular locus of dimension at most $n-m-3$.

Observe that $V(\bfs\Pi_\mathcal{I}):=V(\Pi_{i_1}\klk
\Pi_{i_m})\subset\Pp^{n-1}$ is a set--theoretic complete
intersection, whose singular locus has codimension at least $2$. We
deduce that $V(\bfs\Pi_\mathcal{I})$  is normal and Theorem
\ref{theorem: normal complete int implies irred} shows that it is
absolutely irreducible.

On the other hand, since $\mathrm{pcl}(V)$ is equidimensional of
dimension $n-m$, each irreducible component of
$\mathrm{pcl}(V)\cap\{Y_0=0\}$ has dimension at least $n-m-1$.
Furthermore, $\mathrm{pcl}(V)\cap\{Y_0=0\}$ is contained in the
projective variety $V(\bfs\Pi_\mathcal{I})$, which is absolutely
irreducible of dimension $n-m-1$. We conclude that
$\mathrm{pcl}(V)\cap\{Y_0=0\}$ is also absolutely irreducible of
dimension $n-m-1$, and hence
$$\mathrm{pcl}(V)\cap\{Y_0=0\}=V(\Pi_{i_1}\klk \Pi_{i_m}).$$

Finally, by \cite[Theorem 18.15]{Eisenbud95} we see that
$\Pi_{i_1}\klk \Pi_{i_m}$ define a radical ideal. As a consequence
of the B\'ezout theorem \eqref{eq: geometry: Bezout eq},
$\deg\big(\mathrm{pcl}(V)\cap\{Y_0=0\}\big)= \prod_{j=1}^{m}\deg
\Pi_{i_j}=i_1\cdots i_m$. This finishes the proof of the theorem.
\end{proof}

We conclude this section with a statement that summarizes all the
facts we shall need concerning the projective closure
$\mathrm{pcl}(V)$.
\begin{theorem}\label{theorem: geometry: proj closure of V is abs irred}
Let $p>2$. Then $\mathrm{pcl}(V)\subset\Pp^n$ is a normal
ideal--theoretic complete intersection of dimension $n-m$ and degree
$i_1\cdots i_m$.
\end{theorem}
\begin{proof}
We have already shown that $\mathrm{pcl}(V)$ is equidimensional of
dimension $n-m$ and degree at most $i_1\cdots i_m$. According to
Theorem \ref{theorem: geometry: dimension singular locus}, the
singular locus of $\mathrm{pcl}(V)$ lying in the open set
$\{Y_0\not=0\}$ has dimension at most $n-m-2$, while Lemma
\ref{lemma: geometry: dim singular locus V at infinity} shows that
its singular locus at infinity has dimension at most $n-m-3$. We
conclude that the singular locus of $\mathrm{pcl}(V)$ has dimension
at most $n-m-2$.

Observe that $\mathrm{pcl}(V)$ is contained in the projective
variety $V(\bfs{R}^h):=V(R_j^h:1\le j\le m)$. We have the inclusions
$$
V(\bfs{R}^h)\cap\{Y_0\not=0\} \subset V(\bfs{R}),\quad
V(\bfs{R}^h)\cap\{Y_0=0\}\subset V(\bfs\Pi_\mathcal{I}).
$$
Both $\{R_j:1\le j\le m\}$ and $\{\Pi_{i_j}:1\le j\le m\}$ satisfy
the conditions of the statement of Lemma \ref{lemma: geometry: V is
complete inters}. It follows that $V(\bfs{R})\subset\A^n$ is
equidimensional of dimension $n-m$ and
$V(\bfs\Pi_\mathcal{I})\subset\Pp^{n-1}$ is equidimensional of
dimension $n-m-1$. We conclude that $V(\bfs{R}^h)\subset\Pp^n$ has
dimension at most $n-m$. Taking into account that it is defined by
$m$ polynomials, we deduce that it is a set--theoretic complete
intersection of dimension $n-m$. This implies that it is
equidimensional of dimension $n-m$, and therefore has no irreducible
component contained in the hyperplane at infinity. In particular, it
agrees with the projective closure of its restriction to $\A^n$
(see, e.g., \cite[Proposition I.5.17]{Kunz85}). As such a
restriction is the affine variety $V=V(\bfs{R})$, we deduce that
$$
\mathrm{pcl}(V)=V(\bfs{R}^h).
$$

Since its singular locus has codimension at least $2$, we have that
$V(\bfs{R}^h)$ is a normal set--theoretic complete intersection.
Finally, by \cite[Theorem 18.15]{Eisenbud95} we see that the
polynomials $R_j^h$ $(1\le j\le m)$ define a radical ideal. Then the
B\'ezout theorem \eqref{eq: geometry: Bezout eq} implies
$\deg\mathrm{pcl}\big(V)=\prod_{j=1}^m\deg R_j^h=i_1\cdots i_m$ and
finishes the proof of the theorem.
\end{proof}

%
%---------------------------------------------------------------------
%---------------------------------------------------------------------
%---------------------------------------------------------------------
%---------------------------------------------------------------------
%---------------------------------------------------------------------
%---------------------------------------------------------------------
%---------------------------------------------------------------------
%---------------------------------------------------------------------
%

\section{The number of polynomials in $\mathcal{A}_{\bfs
\lambda}$}\label{section: estimates}
Let $A_r,\ldots,A_{n-1}$ be indeterminates over $\cfq$ and set $\bfs
A:=(A_{n-1},\ldots,A_r)$. Let be given linear forms
$L_1,\ldots,L_m\in\fq[\bfs A]$ which are linearly independent and
$\bfs\alpha:=(\alpha_1\klk\alpha_m)\in\fq^m$. Set
$\bfs{L}:=(L_1,\ldots,L_m)$ and let
$\mathcal{A}:=\mathcal{A}(\bfs{L},\bfs \alpha)$ be the set defined
in the following way:
$$\mathcal{A}:=\left\{T^n+a_{n-1}T^{n-1}\plp a_0\in\fq[T]:
L_j(a_r\klk a_{n-1})+\alpha_j=0\ (1\le j\le m)\right\}.$$
As before, we shall assume that the Jacobian matrix $(\partial \bfs
L/\partial\bfs A)$ is lower triangular in row echelon form and
denote by $1\le i_1<\cdots<i_m\le n-r$ the positions corresponding
to the pivots. Given a factorization pattern
$\bfs\lambda:=1^{\lambda_1}\cdots n^{\lambda_n}$, in this section we
determine the asymptotic behavior of the cardinality of the set
$\mathcal{A}_{\bfs{\lambda}}$ of elements of $\mathcal{A}$ with
factorization pattern $\bfs\lambda$.

For this purpose, in Corollary \ref{coro: fact patterns: systems
pattern lambda} we obtain polynomials $R_1\klk R_m\in\fq[\bfs X]:=
\fq[X_1\klk X_n]$ whose common $\fq$--rational zeros are related to
the quantity $|\mathcal{A}_{\bfs{\lambda}}|$. More precisely, let
$\bfs x:=(\bfs x_{i,j}:1\le i\le n,1\le j\le \lambda_i)\in\fq^n$ be
a common $\fq$--rational zero of $R_1\klk R_m$ of type $\bfs\lambda$
(see Definition \ref{def: fact patterns: type lambda}). We associate
to $\bfs x$ an element $f\in\mathcal{A}_{\bfs\lambda}$ having
$Y_{\ell_{i,j}+k}(\bfs x_{i,j})$ as an $\fqi$--root for $1\le i\le
n$, $1\le j\le\lambda_i$ and $1\le k\le i$. Here, $Y_{\ell_{i,j}+k}$
is the linear form
\begin{equation}\label{eq: estimates: definition Y_ell}
Y_{\ell_{i,j}+k}:=X_{\ell_{i,j}+1}\sigma_{k,i}(\theta_i)+\dots
+X_{\ell_{i,j}+i}\sigma_{k,i}(\theta_i^{q^{i-1}}),
\end{equation}
where $\{\sigma_{k,i}:1\le k\le i\}$ are the elements of the Galois
group $\G_i$ of $\fqi$ over $ \fq$.

Let $\mathcal{A}_{\bfs\lambda}^{sq}:=\{f\in
\mathcal{A}_{\bfs\lambda}: f \mbox{ is square--free}\}$ and let
$\mathcal{A}_{\bfs\lambda}^{nsq}:=\mathcal{A}_\lambda\setminus
\mathcal{A}_{\bfs\lambda}^{sq}$. Corollary \ref{coro: fact patterns:
systems pattern lambda} asserts that any element $f\in
\mathcal{A}_{\bfs\lambda}^{sq}$ is associated with
$w(\bfs\lambda):=\prod_{i=1}^n i^{\lambda_i}\lambda_i!$ common
$\fq$--rational zeros of $R_1\klk R_m$ of type $\bfs\lambda$.
Observe that $\bfs x\in\fq^n$ is of type $\bfs\lambda$ if and only
if $Y_{\ell_{i,j}+k_1}(\bfs x) \neq Y_{\ell_{i,j}+k_2}(\bfs x)$ for
$1\leq i\leq n$, $1\leq j\leq \lambda_i$ and $1\leq k_1 <k_2 \leq
i$. Furthermore, $\bfs x\in\fq^n$ of type $\bfs\lambda$ is
associated with $f\in\mathcal{A}_{\bfs\lambda}^{sq}$ if and only if
$Y_{\ell_{i,j_1}+k_1}(\bfs x) \neq Y_{\ell_{i,j_2}+k_2}(\bfs x)$ for
$1\leq i\leq n$, $1\leq j_1<j_2\leq \lambda_i$ and $1\leq k_1 <k_2
\leq i$. As a consequence, we see that
\begin{align}
\big|\mathcal{A}_{\bfs\lambda}^{sq}\big|= \mathcal{T}(\bfs\lambda)
\big|\big\{\bfs x \in \fq^n: \,R_1(\bfs x)&=\cdots=R_m(\bfs x)=0,\,
Y_{\ell_{i,j_1}+k_1}(\bfs x) \neq
Y_{\ell_{i,j_2}+k_2}(\bfs x)\nonumber  \\
&(1\leq i\leq n,\, 1\leq j_1 < j_2\leq \lambda_i,\,1\leq k_1 <k_2
\leq i)\big\}\big|,\label{eq: estimates: sq with pattern lambda}
\end{align}
where $\mathcal{T}(\bfs\lambda):={1}/w(\bfs\lambda)$. The results of
Section \ref{section: geometry of V} will allow us to establish the
asymptotic behavior of $|\mathcal{A}_{\bfs\lambda}^{sq}|$.
%
%---------------------------------------------------------------------
%---------------------------------------------------------------------
%---------------------------------------------------------------------
%---------------------------------------------------------------------
%
\subsection{The number of $\fq$-rational points of normal
complete intersections}
Let $V\subset\A^n$ be the variety defined by the polynomials
$R_1\klk R_m\in\fq[\bfs X]$ of \eqref{eq: geometry: def R_j}. Denote
by $\mathrm{pcl}(V)\subset\Pp^n$ the projective closure of $V$ and
by
$\mathrm{pcl}(V)^{\infty}:=\mathrm{pcl}(V)\cap\{Y_0=0\}\subset\Pp^{n-1}$
the set of points of $\mathrm{pcl}(V)$ at infinity. Theorems
\ref{theorem: geometry: V at infinity} and \ref{theorem: geometry:
proj closure of V is abs irred} assert that
$\mathrm{pcl}(V)^{\infty}$ and $\mathrm{pcl}(V)$ are
$\fq$--definable normal ideal--theoretic complete intersections of
dimension $n-m-1$ and $n-m$ respectively, both of degree
$\delta_{\bfs L}:=i_1\cdots i_m$.

In what follows, we shall use an estimate on the number of
$\fq$--rational points of a projective normal complete intersection
of \cite{CaMaPr13} (see also \cite{CaMa07} or \cite{GhLa02a} for
other estimates). More precisely, if $W\subset \Pp^n$ is a normal
complete intersection defined over $\fq$ of dimension $n-l\geq 2$,
degree $\delta$ and multidegree $\bfs d:=(d_1,\ldots,d_l)$, then the
following estimate holds (see \cite[Theorem 1.3]{CaMaPr13}):
\begin{equation}\label{eq: estimates: normal var CaMaPr}
\big||W(\fq)|-p_{n-l}\big| \leq
(\delta(D-2)+2)q^{n-l-\frac{1}{2}}+14D^2\delta^2 q^{n-l-1},
\end{equation}
where $p_{n-l}:=q^{n-l}+q^{n-l-1}+\cdots+q+1=|\Pp^{n-l}(\fq)|$ and
$D:=\sum_{i=1}^l(d_i-1)$.

First we estimate the number of $\fq$--rational points of $V$. By
(\ref{eq: estimates: normal var CaMaPr}), we have
\begin{align*}
\big||\mathrm{pcl}(V)(\fq)|-p_{n-m}\big| &\leq
(\delta_{\bfs L}(D_{\bfs L}-2)+2)q^{n-m-\frac{1}{2}}+14D_{\bfs L}^2\delta_{\bfs L}^2q^{n-m-1},\\
\big||\mathrm{pcl}(V)^{\infty}(\fq)|-p_{n-m-1}\big| &\leq
(\delta_{\bfs L}(D_{\bfs L}-2)+2)q^{n-m-\frac{3}{2}}+14D_{\bfs
L}^2\delta_{\bfs L}^2q^{n-m-2},
\end{align*}
where $\delta_{\bfs L}:=i_1\cdots i_m$  and $D_{\bfs
L}:=\sum_{j=1}^m (i_j-1)$. Hence, we obtain
\begin{align}
    \big||V(\fq)|-q^{n-m}\big|& =
    \big||\mathrm{pcl}(V)(\fq)|-|\mathrm{pcl}(V)^{\infty}(\fq)|-
      p_{n-m}+p_{n-m-1}\big|\nonumber\\[1ex]
      & \le
      \big||\mathrm{pcl}(V)(\fq)|-p_{n-m}\big|+
      \big||\mathrm{pcl}(V)^{\infty}(\fq)|-p_{n-m-1}\big|
      \nonumber\\[1ex] &\le   (q+1) q^{n-m-2}\big((\delta_{\bfs L}(D_{\bfs L}-2)+2)
      q^{1/2} +14 D_{\bfs L}^2 \delta_{\bfs L}^2\big). \label{eq: estimates: Vn todos}
  \end{align}

Let $V^{=}$ be the affine subvariety of $V \subset \A^n$ defined by
$$V^{=}:=\mathop{\bigcup_{1\le i\le n}}_{ 1\leq j_1 <j_2\leq
\lambda _i,\, \,  1\leq k_1 <k_2 \leq i}
V\cap\{Y_{\ell_{i,j_1}+k_1}=Y_{\ell_{i,j_2}+k_2}\},
$$
where $Y_{\ell_{i,j}+k}$ are the linear forms of \eqref{eq:
estimates: definition Y_ell}. Then \eqref{eq: estimates: sq with
pattern lambda} shows that $V^=$ represents the vector of
coefficients of roots of non square--free polynomials. Let
$V^{\neq}(\fq):=V(\fq) \setminus V^{=}(\fq)$. Observe that (\ref{eq:
estimates: sq with pattern lambda}) we may reexpressed as
\begin{equation}\label{eq: estimates: A_lambda^sq}
|\mathcal{A}_{\bfs\lambda}^{sq}| =\mathcal{T}(\bfs\lambda)
\big|V^{\neq}(\fq)\big|
\end{equation}
By  the equality $\mathcal{T}(\bfs\lambda)\big|V^{\neq}(\fq)\big|
=\mathcal{T}(\bfs\lambda)\big|V(\fq)\big|-\mathcal{T}(\bfs\lambda)\big|V^{=}(\fq)\big|$
and \eqref{eq: estimates: A_lambda^sq}, we obtain
\begin{align}
\big||\mathcal{A}_{\bfs\lambda}|
-\mathcal{T}(\bfs\lambda)\,q^{n-m}\big|&=
\big||\mathcal{A}_{\bfs\lambda}^{sq}|+
|\mathcal{A}_{\bfs\lambda}^{nsq}|-\mathcal{T}(\bfs\lambda)q^{n-m}\big|\nonumber\\
&=\big|\mathcal{T}(\bfs\lambda)\,|V^{\neq}(\fq)|
+|\mathcal{A}_{\bfs\lambda}^{nsq}|-\mathcal{T}(\bfs\lambda)q^{n-m}\Big|\nonumber\\
&=\big|\mathcal{T}(\bfs\lambda)\,
|V(\fq)\big|-\mathcal{T}(\bfs\lambda)|V^{=}(\fq)|
+|\mathcal{A}_{\bfs\lambda}^{nsq}|-\mathcal{T}(\bfs\lambda)q^{n-m}\big|\nonumber\\
& \leq \mathcal{T}(\bfs\lambda)\,\big||V(\fq)|-q^{n-m}\big|+
\big||\mathcal{A}_{\bfs\lambda}^{nsq}|-\mathcal{T}(\bfs\lambda)|V^{=}(\fq)|\big|.
\label{eq: estimates: estimate A_lambda}
\end{align}
The first term in the right--hand side is bounded using (\ref{eq:
estimates: Vn todos}). On the other hand,
\begin{equation}\label{eq: estimates: rough bound A_lambda^nsq}
\big||\mathcal{A}_{\bfs\lambda}^{nsq}|
-\mathcal{T}(\bfs\lambda)|V^{=}(\fq)|\big| =|\mathcal{A}_{\bfs
\lambda}^{nsq}|-\mathcal{T}(\bfs\lambda)\, |V^{=}(\fq)| \leq
|\mathcal{A}_{\bfs \lambda}^{nsq}|.
\end{equation}
As a consequence, by \eqref{eq: estimates: Vn todos}, \eqref{eq:
estimates: estimate A_lambda} and \eqref{eq: estimates: rough bound
A_lambda^nsq} we obtain the following result.
\begin{theorem}\label{theorem: estimates: bound without discr locus}
For $p>2$, $q > n$ and $3\le r\le n-m$, we have
\begin{align*}
\big||\mathcal{A}_{\bfs\lambda}|
&-\mathcal{T}(\bfs\lambda)\,q^{n-m}\big|\leq\\
&\le (q+1)\,q^{n-m-2}\mathcal{T}(\bfs\lambda)\Big(\big(\delta_{\bfs
L}(D_{\bfs L}-2)+2\big)q^{\frac{1}{2}} +14\,D_{\bfs L}^2
\delta_{\bfs L}^2\Big)+|\mathcal{A}_{\bfs \lambda}^{nsq}|,
\end{align*}
where $\delta_{\bfs L}:=i_1\cdots i_m$  and $D_{\bfs
L}:=\sum_{j=1}^m (i_j-1)$.
\end{theorem}
%
%---------------------------------------------------------------------
%---------------------------------------------------------------------
%---------------------------------------------------------------------
%---------------------------------------------------------------------
%
%\subsection{Upper bounds on the number of polynomials in the discriminant locus}
%
It remains to obtain an upper bound on the number
$|\mathcal{A}_{\bfs\lambda}^{nsq}|$ of polynomials in
$\mathcal{A}_{\bfs\lambda}$
which are not square-free. %In this section, we obtain two different
%upper bounds on $|\mathcal{A}^{nsq}|$. The first bound is preferable
%when $\delta_{\bfs L}$ is high and $T(\bfs\lambda)$ is not
%sufficiently small, while the second one may be convenient in the
%remaining cases.
%
%---------------------------------------------------------------------
%---------------------------------------------------------------------
%
%\subsubsection{A simple bound}
%
To this end, observe that a polynomial $f\in \mathcal{A}$ is not
square--free if and only if its discriminant is equal to zero. In
\cite{MaPePr13} we study the so--called {\em discriminant locus} of
$\mathcal{A}$, namely the set $\mathcal{A}^{nsq}$ formed by the
elements of $\mathcal{A}$ whose discriminant is equal to zero (see
also \cite{FrSm84} for further results on discriminant loci).
According to \cite[Theorem A.3]{MaPePr13}, the discriminant locus
$\mathcal{A}^{nsq}$ is the set of $\fq$--rational points of a
hypersurface of degree $n(n-1)$ of a suitable $(n-m)$--dimensional
affine space. As a consequence, by, e.g., \cite[Lemma 2.1]{CaMa06},
we have
\begin{equation}\label{eq: estimates: upper bound discr locus}
|\mathcal{A}_{\bfs\lambda}^{nsq}|\le |\mathcal{A}^{nsq}|\leq n(n-1)
\,q^{n-m-1}.
\end{equation}
Combining Theorem \ref{theorem: estimates: bound without discr
locus} (with a slightly simplified bound) and (\ref{eq: estimates:
upper bound discr locus}) we deduce the following result.
\begin{theorem}\label{theorem: estimates: bound with discr locus 1}
For $p>2$, $q > n$ and $3\le r\le n-m$, we have
\begin{equation}\label{eq: estimates: estimate 1}
\big||\mathcal{A}_{\bfs\lambda}|
-\mathcal{T}(\bfs\lambda)\,q^{n-m}\big|\leq
q^{n-m-1}\big(2\,\mathcal{T}(\bfs\lambda)\,D_{\bfs L}\delta_{\bfs
L}q^{\frac{1}{2}} +19\,\mathcal{T}(\bfs\lambda)\,D_{\bfs L}^2
\delta_{\bfs L}^2+n(n-1)\big).
\end{equation}
\end{theorem}
%
%---------------------------------------------------------------------
%---------------------------------------------------------------------
%
\subsection{Factorization patterns of polynomials with prescribed
coefficients}
In this section we briefly indicate how Theorem \ref{theorem:
estimates: bound with discr locus 1} is applied when $\mathcal{A}$
consists of the polynomials of $\mathcal{P}$ with certain prescribed
coefficients. Given $0< i_1<i_2<\cdots< i_m\le n$ and $\bfs
\alpha:=(\alpha_{i_1}\klk \alpha_{i_m})\in\fq^m$, set
$\mathcal{I}:=\{i_1\klk i_m\}$ and
\begin{equation}\label{eq: estimates: family A^m}
\mathcal{A}^m:= \mathcal{A}^m(\mathcal{I},\bfs
\alpha):=\left\{T^n+a_1T^{n-1}\plp a_n\in\fq[T]:
a_{i_j}=\alpha_{i_j}\ (1\le j\le m)\right\}.
\end{equation}
For a given factorization pattern $\bfs\lambda$, let $G\in\fq[\bfs
X,T]$ be the polynomial of \eqref{eq: fact patterns: pol G}.
According to Lemma \ref{lemma: fact patterns: sym pols and pattern
lambda}, an element $f\in \mathcal{A}^m$ has factorization pattern
$\bfs\lambda$ if and only if there exists $\bfs x$ of type
$\bfs\lambda$ such that
\begin{equation}\label{eq: estimates: polynomial syst for presc coeff}
(-1)^{i_j}\Pi_{i_j}(\bfs Y(\bfs x))=\alpha_{i_j}\quad(1\le j\le m).
\end{equation}
Therefore, applying Theorem \ref{theorem: estimates: bound with
discr locus 1} with $\delta_{\mathcal{I}}:=i_1\cdots i_m$ and
$D_{\mathcal{I}}:=\sum_{j=1}^m(i_j-1)$, we obtain the following
result. %We shall assume that, if $i_m=n$, then $\alpha_n\not=0$.
\begin{corollary} For $p>2$, $q > n$ and $i_m\le n-3$, we
have
$$
\big||\mathcal{A}^m_{\bfs
\lambda}|-\mathcal{T}(\bfs\lambda)\,q^{n-m}\big|\le
q^{n-m-1}\big(2\,\mathcal{T}(\bfs\lambda)\,D_{\mathcal{I}}\,
\delta_{\mathcal{I}}\,q^{\frac{1}{2}}
+19\,\mathcal{T}(\bfs\lambda)\,D_{\mathcal{I}}^2\,
\delta_{\mathcal{I}}^2+n(n-1)\big). $$
\end{corollary}
%
%---------------------------------------------------------------------
%---------------------------------------------------------------------
%---------------------------------------------------------------------
%---------------------------------------------------------------------
%---------------------------------------------------------------------
%---------------------------------------------------------------------
%---------------------------------------------------------------------
%---------------------------------------------------------------------
%
\section{The number of polynomials in $\mathcal{A}_{\bfs \lambda}$
in the sparse case}
In \cite{CaMaPr12}, \cite{CeMaPePr14} and \cite{MaPePr13} a
methodology to deal with combinatorial problems over finite fields
is developed. It is based on the fact that many combinatorial
problems can be described by means of symmetric polynomials, and
varieties defined by symmetric polynomials have particular features
that can be exploited in order to obtain ``good'' estimates on their
number of $\fq$--rational points. A new avatar of these assertions
can be seen in Sections \ref{section: geometry of V} and
\ref{section: estimates} above.

In particular, similar techniques as in \cite{CeMaPePr14} can be
applied in order to obtain a further estimate on the number of
elements in $\mathcal{A}_{\bfs \lambda}$, which holds when the
linear forms $L_1\klk L_m$ are ``sparse''. More precisely, if
$L_1\klk L_m$ are linearly independent elements of $\fq[A_r\klk
A_{n-1}]$ with $r\ge m+2$, then we will able to show that
$$\big||{\mathcal A_{\bfs \lambda}}|-\mathcal{T}(\bfs \lambda)
\,q^{n-m}\big|={\mathcal O}(q^{n-m-1}),$$
improving thus the ${\mathcal O}(q^{n-m-\frac{1}{2}})$ estimate  of
the left--hand side of Theorem \ref{theorem: estimates: bound with
discr locus 1}. We remark that this estimate is valid without
restrictions on the characteristic of $\fq$.

As the arguments differ slightly from the ones of \cite{CeMaPePr14}
we shall merely sketch the approach. Roughly speaking, the results
of \cite{CeMaPePr14} allows us to deduce that the singular locus of
the variety $V$ defined by the polynomials of \eqref{eq: geometry:
def R_j} has codimension at least $3$. Combining this with results
about the geometry of $V$ of Section \ref{section: geometry of V} we
conclude that its projective closure $\mathrm {pcl}(V)$ is regular
in codimension 2. Our estimate then follows from estimates on the
number of $\fq$--rational points of complete intersections which are
regular in codimension 2 due to \cite{CaMaPr13}.
%
%---------------------------------------------------------------------
%---------------------------------------------------------------------
%---------------------------------------------------------------------
%---------------------------------------------------------------------
%
\subsection{The geometry of the set of zeros of $R_1\klk R_m$ for large $r$}
Let be given $r$ with $m+2 \leq r \leq n-m$, let $A_{r}\klk A_{n-1}$
be indeterminates over $\cfq$ and let $L_1,\dots,L_m$ be linear
forms of $\fq[A_{r}\klk A_{n-1}]$ which are linearly independent.
For $\bfs\alpha:=(\alpha_1\klk\alpha_m)\in\fq^m$, we set
$\bfs{L}:=(L_1,\ldots,L_m)$ and consider as before the linear
variety $\mathcal{A}:=\mathcal{A}(\bfs{L},\bfs\alpha)$ defined as
$$\mathcal{A}:=\left\{T^n+a_{n-1}T^{n-1}\plp a_0\in\fq[T]:
\bfs L(a_r\klk a_{n-1})+\bfs \alpha=\bfs 0\right\}.$$

Let $R_1\klk R_m$ be the polynomials of $\fq[\bfs X]:=\fq[X_1\klk
X_n]$ defined as
\begin{equation}\label{eq: prescribed coeff: Rj}
R_j:=S_j\big(\Pi_1,\ldots,\Pi_{n-r}\big)\quad (1\le j\le m)
\end{equation}
where $\Pi_1 \klk \Pi_{n-r}$ are the first $n-r$ elementary
symmetric polynomials of \linebreak $\fq[Y_1 \klk Y_n]$, $\bfs
Y:=(Y_1\klk Y_n)$ is the vector of linear forms of $\cfq[\bfs X]$
defined as in (\ref{eq: fact patterns: def linear forms Y}) and the
linear polynomials $S_1\klk S_m\in \cfq[Z_1\klk Z_{n-r}]$ are
defined as in (\ref{eq: geometry: def R_j}). According to Corollary
\ref{coro: fact patterns: systems pattern lambda}, we can express
the number of elements of $\mathcal{A}_{\bfs \lambda}$ in terms of
the number of $\fq$--rational points of the variety $V\subset \A^n$
defined by $R_1\klk R_m$.

As in the previous sections, we assume that the Jacobian matrix
$(\partial \bfs L/\partial \bfs A)$ of the vector of linear forms
$\bfs L$ with respect to $\bfs A:=(A_{n-1}\klk A_r)$ is lower
triangular in row--echelon form. Hence, there exist $1\leq i_1< i_2<
\dots< i_m \leq n-r $ such that $\deg R_j=i_j$ $(1 \leq j \leq m)$.
As before, the numbers
\begin{equation}\label{eq: prescribed coeff: def deltaL}
\delta_{\bfs L}:=i_1\cdots i_m \quad \hbox{ and } \quad D_{\bfs
L}:=\sum_{j=1}^m (i_j-1)
\end{equation}
will play a central role in our estimates.

Next we show that the projective closure $\mathrm{pcl}(V)$ of $V$
and the set $\mathrm{pcl}(V)^{\infty}$ of points of
$\mathrm{pcl}(V)$ at infinity are complete intersections which are
regular in codimension two. For this purpose, we rely on results on
the geometry of complete intersections defined by symmetric
polynomials of \cite[Section 3]{CeMaPePr14}.
%In order to prove that the variety $V$
%defined by the polynomials $R_1\klk R_m$ of \eqref{eq: prescribed
%coeff: Rj} and its  geometric features of and of we shall use the
%results proved in \cite{CeMaPePr14} for a ``general'' variety
%$\mathcal V$ , which are summarized below.
%
\begin{theorem}\label{theorem: prescribed coeff: complete int and sing locus}
Let be given positive integers $m$, $r$ and $n$ with $q>n$ and
$m+2\leq r\leq n-m$. %Let $S_1 \klk S_m$ be polynomials in $\fq[Z_1 \klk
%Z_{n-r}]$. Assume that they satisfy the following conditions:
%\begin{enumerate}[(\,a)]
%\item $S_1\klk S_m$ form a regular sequence in  $\fq[Z_1\klk
% Z_{n-r}]$.
% \item The Jacobian matrix $\left(\partial {\bfs S}/\partial{\bfs
% Z}\right)({\bfs z})$ has full rank $m$ for every ${\bfs z}\in {\A}^{n-r}$.
%\end{enumerate}
Let $R_1\klk R_m$ be the polynomials of (\ref{eq: prescribed coeff:
Rj}) and $V\subset\A^n$ the affine variety defined by $R_1\klk R_m$.
Then
 \begin{enumerate}[(\,1)]
\item $V$ is an ideal--theoretic complete intersection
of dimension $n-m$ and $\deg(V) \leq \prod_{i=1}^m \deg(R_i)$.
\item  The set of points $\bfs y\in \A^n$ for which
$\left(\partial {\bfs R}/\partial{\bfs Y}\right)({\bfs y})$ has not
full rank, has dimension at most $n-r-1$. In particular, the
singular locus of $V$ has dimension at most $n-r-1$.
\end{enumerate}
\end{theorem}

\begin{proof}%[Sketch of proof.]
As $S_1\klk S_m\in\fq[Z_1\klk Z_{n-r}]$ are linear polynomials which
are linearly independent, the hypotheses (H1) and (H2) of
\cite[Section 3.2]{CeMaPePr14} are satisfied. Then \cite[Theorem
3.2]{CeMaPePr14} shows the second assertion.  %is
%proved in Theorem 7 in \cite{CeMaPePr14} which provides estimates on
%the dimension of set
%%
%$$\mathcal{J}_{\boldsymbol{R}}:=
%\left\{\boldsymbol{x}\in \A^n:\ \left(\partial {\bfs
%R}/\partial{\bfs X}\right)({\bfs x})<m \right\},$$
%%
%when $m\leq r\leq n-m$ and without restrictions on the
%characteristic of $\fq$. In fact, if $\boldsymbol{x}\in
%\mathcal{J}_{\boldsymbol{R}}$, then  $\mathrm{rank}\left(\partial
%{\bf \Pi}/\partial{\bf X}\right)({\bfs x})<m $ with $\bf\Pi:=(\Pi_1
%\klk \Pi_{n-r}).$ This implies that the points $\boldsymbol{x}\in
%\mathcal{J}_{\boldsymbol{R}} $ have at most $n-r-1$
%pairwise-distinct coordinates, and thus
%$\mathcal{J}_{\boldsymbol{R}}$ is an affine variety of dimension at
%most $n-r-1$.
%
%Now we address the first assertion.  Lemma \ref{lemma: geometry: V
%is complete inters} asserts that $\mathcal{V}\subset \A^n$ is an
%set--theoretic complete intersection of dimension $n-m$. We deduce
%the fact that $\mathcal{V}$ is ideal--theoretic complete
%intersection and the bound on its degree from estimates on the
%dimension of set $\mathcal{J}_{\boldsymbol{R}}$. In fact, this proof
%follows the same lines of the proof of Corollary \ref{coro:
%geometry: radicality and degree Vr} in this paper with the estimates
%on the dimension of $\mathcal{J}_{\boldsymbol{R}}$ provided in the
%second part of this theorem.
Finally, from \cite[Corollary 3.3]{CeMaPePr14} we readily deduce the
first assertion of the theorem.
\end{proof}

If the polynomials $R_1\klk R_m$ of (\ref{eq: prescribed coeff: Rj})
are homogeneous (for example, if $R_j=\Pi_{i_j}$ for $1\le j\le m$),
we may somewhat strengthen the conclusions of Theorem \ref{theorem:
prescribed coeff: complete int and sing locus}, as the next result
asserts.
\begin{corollary}\label{coro: prescribed coeff: complete int and sing locus}
Let notations and assumptions be as above. Suppose further that $R_1
\klk R_m$ are homogeneous. Then
\begin{enumerate}[(\,1)]
\item The projective variety $\mathcal{V}\subset \Pp^{n-1}$
defined by $R_1 \klk R_m$ is an ideal--theoretic
complete intersection of dimension $n-m-1$ and $\deg(\mathcal{V})=
\prod_{i=1}^m \deg(R_i)$.
\item The set of points $\bfs y\in \Pp^{n-1}$ for which
$\left(\partial {\bfs R}/\partial{\bfs Y}\right)(\bfs y)$ has not
full rank, has dimension at most $n-r-2$. In particular, the
singular locus of $\mathcal{V}$ has dimension at most $n-r-2$.
\end{enumerate}
\end{corollary}
\begin{proof}%[Sketch of proof.]
The second assertion readily follows from that of Theorem
\ref{theorem: prescribed coeff: complete int and sing locus}. On the
other hand, the first assertion of Theorem \ref{theorem: prescribed
coeff: complete int and sing locus} implies that $R_1 \klk R_m$ form
a regular sequence and define a radical ideal. Therefore,
$\mathcal{V}$ is an ideal--theoretic complete intersection of
dimension $n-m-1$.
%Furthermore, the set of points $\boldsymbol{x}\in \mathcal{V}$ for
%which $\left(\partial {\bfs R}/\partial{\bfs X}\right)({\bfs x})$
%has not full rank, has codimension at least $(n-m-1)-(n-r-2)\ge 3$
%in $\mathcal{V}$. an ideal--theoretic complete intersection of dimension $n-m-1$ with
As a consequence, the B\'ezout theorem \eqref{eq: geometry: Bezout
eq} shows that $\deg(\mathcal{V}) = \prod_{i=1}^m \deg(R_i)$.
%
%To finish the proof of corollary we must prove that
%$\mathcal{V}\subset \Pp^{n-1}$ is an absolutely irreducible variety.
%Since $r \geq m+2$, the singular locus of
%$\mathrm{pcl}(\mathcal{V})$ has codimension at least $3$. Therefore,
%by the Hartshorne connectedness theorem (see, e.g., \cite[Theorem
%VI.4.2]{Kunz85}) we conclude that $\mathrm{pcl}(\mathcal{V})$ is an
%absolutely irreducible variety.
\end{proof}

As before, we consider the projective closure
$\mathrm{pcl}(V)\subset\mathbb{P}^n$ of the affine variety
$V\subset\A^n$ defined by the polynomials $R_1\klk R_m$ of (\ref{eq:
prescribed coeff: Rj}) with respect to the embedding of $\A^n$ into
$\Pp^n$ defined as $(y_1,\dots,y_n)\mapsto(1:y_1:\dots:y_n)$. We
also denote by $\mathrm{pcl}(V)^\infty$ the set of points of
$\mathrm{pcl}(V)$ at infinity, namely
$\mathrm{pcl}(V)^\infty:=\mathrm{pcl}(V) \cap\{Y_0=0\}$. In
connection with the latter, we observe that, if $R_j^h$ is the
homogenization of $R_j$ defined as in (\ref{eq: geometry:
homogenizacion de R_j}), the polynomials $R_j^h(0,Y_1\klk
Y_n)=\Pi_{i_j}$ $(1\le j\le m)$ satisfy the hypotheses of Corollary
\ref{coro: prescribed coeff: complete int and sing locus}. The next
theorem summarizes all the properties of $\mathrm{pcl}(V)$ and
$\mathrm{pcl}(V)^\infty$ which are relevant for our purposes.
\begin{theorem} \label{theorem: prescribed coeff: geometry}
Let $n$, $m$ and  $r$ be positive integers with $q>n$ and $m+2 \leq
r \leq n-m$. The projective varieties $\mathrm{pcl}(V)\subset\Pp^n$
and $\mathrm{pcl}(V)^{\infty}:=\mathrm{pcl}(V)\cap\{Y_0=0\}\subset
\Pp^{n-1}$  are ideal--theoretic complete intersections defined over
$\fq$, of dimension $n-m$ and $n-m-1$ respectively, both of degree
$\delta_{\boldsymbol{L}}$, having singular locus of dimension at
most $n-r-1$ and $n-r-2$ respectively.
\end{theorem}
\begin{proof}[Sketch of the proof]
The proof of the assertions concerning $\mathrm{pcl}(V)$ follows the
lines of that of Theorem \ref{theorem: geometry: proj closure of V
is abs irred}, applying Theorem \ref{theorem: prescribed coeff:
complete int and sing locus} and Corollary \ref{coro: prescribed
coeff: complete int and sing locus} instead of Theorem \ref{theorem:
geometry: dimension singular locus} and Lemma \ref{lemma: geometry:
dim singular locus V at infinity}. On the other hand, the assertions
about $\mathrm{pcl}(V)^\infty$ are shown following the proof of
Theorem \ref{theorem: geometry: V at infinity}, using Corollary
\ref{coro: prescribed coeff: complete int and sing locus} instead of
Lemma \ref{lemma: geometry: V is complete inters}, Theorem
\ref{theorem: geometry: dimension singular locus} and Lemma
\ref{lemma: geometry: dim singular locus V at infinity}.
\end{proof}
%
%---------------------------------------------------------------------
%---------------------------------------------------------------------
%---------------------------------------------------------------------
%---------------------------------------------------------------------
%
\subsection{The estimate of $|\mathcal A_{\bfs \lambda}|$ for large $r$}
In what follows, we shall use an estimate on the number of
$\fq$--rational points of a projective singular complete
intersection defined over $\fq$ due to \cite{CaMaPr13} (see
\cite{GhLa02a} for similar estimates). More precisely, if
$W\subset\Pp^n$ is an $\fq$--definable ideal--theoretic complete
intersection of dimension $n-l$, degree $\delta\ge 2$, multidegree
$(d_1\klk d_l)$ and singular locus of dimension at most $n-l-3$,
then the following estimate holds (see \cite[Corollary
8.4]{CaMaPr13}):
\begin{equation}\label{eq: estimates: rat points CML}
\big||W(\fq)|-p_{n-l}\big|\leq 14 D^3\delta^2q^{n-l-1},
\end{equation}
where $D:=\sum_{i=1}^l(d_i-1)$.

As before, let be given positive integers $m$, $n$ and $r$ with
$q>n$ and $m+2\le r\le n-m$. Let $V\subset\A^n$ be the variety
defined by the polynomials $R_1\klk R_m$ of \eqref{eq: prescribed
coeff: Rj} and let $\delta_{\bfs L}$ and $D_{\bfs L}$ be defined as
in \eqref{eq: prescribed coeff: def deltaL}. Theorem \ref{theorem:
prescribed coeff: geometry} shows that the projective closure
$\mathrm{pcl}(V)$ of $V$ and its set of points at infinity
$\mathrm{pcl}(V)^{\infty}$ satisfy all the requirements of
\cite[Corollary 8.4]{CaMaPr13}. Then \eqref{eq: estimates: rat
points CML} implies
\begin{align*}
\big||\mathrm{pcl}(V)(\fq)|-p_{n-m}\big|\le& 14 D_{\boldsymbol{L}}^3
\delta_{\boldsymbol{L}}^2q^{n-m-1},  \\
\big||\mathrm{pcl}(V)^{\infty}(\fq)|-p_{n-m-1}\big|\le& 14
D_{\boldsymbol{L}}^3 \delta_{\boldsymbol{L}}^2q^{n-m-2}.
\end{align*}

Arguing as in \eqref{eq: estimates: Vn todos} we obtain
\begin{align}
    \big||V(\fq)|-q^{n-m}\big|& \le
      \big||\mathrm{pcl}(V)(\fq)|-p_{n-m}\big|+
      \big||\mathrm{pcl}(V)^{\infty}(\fq)|-p_{n-m-1}\big|
      \nonumber\\[1ex] &\le (q+1)14 D_{\bfs L}^3 \delta_{\bfs L}^2
q^{n-m-2}\le 21 D_{\bfs L}^3 \delta_{\bfs L}^2 q^{n-m-1}. \label{eq:
prescribed coeff: estimate aux}
  \end{align}
We are now ready to state the main result of the section.
\begin{theorem}\label{theorem: prescribed coeff: estimate 2}
For $q>n$ and $m+2\le r \le n-m$, we have
$$
\big||\mathcal{A}_{\bfs\lambda}|
-\mathcal{T}(\bfs\lambda)\,q^{n-m}\big|\leq
q^{n-m-1}\big(21\,\mathcal{T}(\bfs\lambda)\, D_{\bfs
L}^3\delta_{\bfs L}^2+ n(n-1)\big).
$$
\end{theorem}
\begin{proof}
Combining \eqref{eq: estimates: estimate A_lambda} and \eqref{eq:
estimates: rough bound A_lambda^nsq} with \eqref{eq: prescribed
coeff: estimate aux}, we see that
\begin{align*}
\big||\mathcal{A}_{\bfs\lambda}|
-\mathcal{T}(\bfs\lambda)\,q^{n-m}\big|&\leq
\mathcal{T}(\bfs\lambda)\,\big||V(\fq)|-q^{n-m}\big|+
|\mathcal{A}_{\bfs\lambda}^{nsq}|\\
&\leq \mathcal{T}(\bfs\lambda)\,21 D_{\bfs L}^3 \delta_{\bfs L}^2
q^{n-m-1}+ |\mathcal{A}_{\bfs\lambda}^{nsq}|.
\end{align*}
The statement of the theorem follows immediately from \eqref{eq:
estimates: upper bound discr locus}.
\end{proof}

Finally, we apply Theorem \ref{theorem: prescribed coeff: estimate
2} to any family consisting of the elements of $\mathcal{P}$ with
certain prescribed coefficients. More precisely, let
$\mathcal{A}^m:= \mathcal{A}^m(\mathcal{I},\bfs \alpha)$ be the
family of polynomials of \eqref{eq: estimates: family A^m}.
For a given factorization pattern $\bfs\lambda$ and the polynomial
$G\in\fq[\bfs X,T]$ of \eqref{eq: fact patterns: pol G}, by
\eqref{eq: estimates: polynomial syst for presc coeff} we see that
$f:=G(\bfs x,T)$ belongs to $\mathcal{A}^m_{\bfs \lambda}$ if and
only if there exists $\bfs x$ of type $\bfs\lambda$ with
$(-1)^{i_j}\Pi_{i_j}(\bfs Y(\bfs x))=\alpha_{i_j}$ for $1\le j\le
m$. If $\delta_{\mathcal{I}}:=i_1\cdots i_m$ and
$D_{\mathcal{I}}:=\sum_{j=1}^m(i_j-1)$, then by Theorem
\ref{theorem: prescribed coeff: estimate 2} we obtain the following
result. %We shall assume that, if $i_m=n$, then $\alpha_n\not=0$.
\begin{corollary} If $q > n$ and $i_m\le n-m-2$, then we
have
$$\big||\mathcal{A}^m_{\bfs
\lambda}|-\mathcal{T}(\bfs\lambda)\,q^{n-m}\big|\le
q^{n-m-1}\big(21\,\mathcal{T}(\bfs\lambda)\,D_{\mathcal{I}}^3\,
\delta_{\mathcal{I}}^2 + n(n-1)\big).$$
\end{corollary}

Comparing the estimates of Theorems \ref{theorem: estimates: bound
with discr locus 1} and \ref{theorem: prescribed coeff: estimate 2},
we observe that the latter shows that $|\mathcal{A}_{\bfs
\lambda}|=\mathcal{T}(\bfs \lambda)\,q^{n-m}
+\mathcal{O}(q^{n-m-1})$, while the former only asserts that
$|\mathcal{A}_{\bfs \lambda}|=\mathcal{T}(\bfs \lambda)\,q^{n-m}
+\mathcal{O}(q^{n-m-1/2})$. Indeed, for $q \geq (11 D_{\bfs
L}^2\delta_{\bfs L})^2$ the upper bound for $\big||{\mathcal A_{\bfs
\lambda}}|-\mathcal{T}(\bfs \lambda)\,q^{n-m}\big|$ of Theorem
\ref{theorem: prescribed coeff: estimate 2} is smaller than that of
Theorem \ref{theorem: estimates: bound with discr locus 1}.
Furthermore, Theorem \ref{theorem: prescribed coeff: estimate 2}
holds without any restriction on the characteristic $p$ of $\fq$,
while Theorem \ref{theorem: estimates: bound with discr locus 1} is
valid only for $p>2$. On the other hand, Theorem \ref{theorem:
estimates: bound with discr locus 1} allows a larger range of values
of $m$, namely $1 \le m \le n-3$, while Theorem \ref{theorem:
prescribed coeff: estimate 2} requires that $1 \le m \le n/2-1$. We
may summarize these remarks by saying that both results are somewhat
complementary.

\bibliographystyle{alpha}
\bibliography{refs1,finite_fields,Newref}

\end{document}